\documentclass[12pt,reqno,a4paper]{amsart} 
\usepackage{graphicx}
\usepackage{amssymb,amsmath}
\usepackage{amsthm}
\usepackage{color,graphicx}
\usepackage{hyperref}
\usepackage{color}
\usepackage{mathabx}
\usepackage{epstopdf}

\usepackage{subfigure} 

\usepackage{verbatim}

\newcommand{\be}{\begin{equation}}

\newcommand{\ee}{\end{equation}}

\newtheorem{lema}{Lemma}[section]
\newtheorem{prop}{Proposition}[section]
\newtheorem{defi}{Definition}[section]

\newcommand{\cn}{{\rm \,cn}}
\newcommand{\sn}{{\rm \,sn}}
\newcommand{\dn}{{\rm \,dn}}

\newcommand{\Ker}{{\rm \,Ker}}

\numberwithin{equation}{section}
\numberwithin{figure}{section}

\newtheorem{theorem}{Theorem}[section]
\newtheorem{proposition}[theorem]{Proposition}
\newtheorem{remark}[theorem]{Remark}
\newtheorem{lemma}[theorem]{Lemma}

\begin{document}
\vglue-1cm \hskip1cm
\title[Periodic waves for the sinh-Gordon equation]{Qualitative Aspects of Periodic Traveling Waves for the Sinh-Gordon equation} 

\begin{center}
\thanks{B.S. Lonardoni is supported by CAPES/Brazil - Finance Code 001. F. Natali is partially supported by CNPq/Brazil (grant 303907/2021-5).
}
\maketitle

{\bf Beatriz Signori Lonardoni \quad \qquad \ \ \ F\'abio Natali \ }\\

{\quad \quad \ \ beatrizslonardoni@gmail.com \quad \quad \quad \quad  \ fmanatali@uem.br \quad}

\vspace{2mm}

{Departamento de Matem\'atica - Universidade Estadual de Maring\'a\\
	Avenida Colombo, 5790, CEP 87020-900, Maring\'a, PR, Brazil}\\

\vspace{3mm}
\end{center}

\begin{abstract}
This paper presents a comprehensive analysis of several aspects of the sinh-Gordon equation within a periodic setting. Our investigation proceeds in three main stages. First we establish the existence of periodic solutions for a fixed wave speed and varying periods by applying the mountain pass theorem. Subsequently, for a fixed period, we construct a family of periodic solutions that depend smoothly on the wave speed; this is achieved via the implicit function theorem. The spectral stability of these waves is then rigorously addressed. We perform a detailed spectral analysis of the linearized operator around the wave of fixed period. A key element in this analysis is the monotonicity of the period map, which, when combined with Morse index theory, enables us to fully characterize the non-positive spectrum of the projected operator in the space of zero-mean periodic functions. Finally, by employing the Hamiltonian-Krein index analysis, we determine the spectral stability and instability of the constructed waves. We additionally discuss qualitative aspects of the Cauchy problem associated with the sinh-Gordon equation, including local well-posedness and blow-up phenomena. The former supports a new linearization of the problem, while the latter predicts the spectral instability of the wave.
\end{abstract}

\section{Introduction} 

In this article, we investigate qualitative properties regarding the sinh-Gordon equation, which is given as 
\begin{equation}\label{KF2}
u_{tt}-u_{xx}+\alpha\sinh(u)=0,
\end{equation}
where $\alpha\in\mathbb{R}\backslash\{0\}$ and $u:\mathbb{R} \times \mathbb{R}_+ \rightarrow \mathbb{R}$ is $L$-periodic in the spatial variable, that is, $u(x+L,t)=u(x,t)$ for all $t\geq0$. Equation $(\ref{KF2})$ is an important model in the theory of nonlinear partial differential equations in mathematical physics, arising in solid-state physics, differential geometry, integrable systems, fluid dynamics, relativistic quantum mechanics, and other physical phenomena (see \cite{dodd}, \cite{drazin}, \cite{greiner}, \cite{perring-skyrme}, and references therein). \\
\indent The equation \eqref{KF2} can be considered as a toy model of the integrable equation
\begin{equation}\label{KF2'}
	u_{tt} - u_{xx} + \sinh(u) = 0,
\end{equation}
that is, when $\alpha=1$ in $(\ref{KF2})$. This specific equation has positive energy, which allows us to obtain global-in-time solutions in the energy space $H_{per}^1 \times L_{per}^2$. In addition, one may expect orbital and spectral stability results for traveling waves, since they arise as critical points of a suitable Lyapunov functional. When $\alpha=-1$, the energy associated with $(\ref{KF2})$ becomes sign-indefinite, in contrast with the positive-definite energy of $(\ref{KF2'})$. For this reason, the case $\alpha=-1$ may be regarded as a \textit{nonlinear Klein–Gordon equation}, while the case $\alpha=1$ corresponds to a $\phi^{2n}-$ \textit{equation}. This interpretation is justified by the Taylor expansion of the hyperbolic sine function, $\sinh(u) = u + \frac{u^3}{6} + \mathcal{O}(u^5)$, which reveals the underlying polynomial structure of the nonlinearity.

The development of analytical techniques for this class of equations remain central themes in the field, motivating research into the qualitative and quantitative properties of the sinh–Gordon model and its generalized forms. Although the sinh–Gordon equation shares many features with other integrable models, its exponential nonlinearity introduces additional analytical challenges, especially when the case $\alpha=-1$ is considered. These challenges include more delicate stability properties, more complex localized structures, and the possibility of blow-up behaviour.

In this setting, two main lines of research emerge: the construction of solutions through algebraic and numerical methods, and the investigation of the stability of these analytical solutions. For the first question, we consider periodic traveling wave solutions of the form $u(x,t)=\varphi(x-ct)$ with wave speed $c\in \mathbb{R}$. Substituting this ansatz into equation $(\ref{KF2})$, we obtain the nonlinear second-order ordinary differential equation
\begin{equation}\label{travSG33}
(1-c^2)\varphi'' - \alpha\sinh(\varphi) = 0.
\end{equation}
A planar analysis shows that $(0,0)$ is a center, which ensures the existence of periodic solutions if and only if $c \neq \pm 1$ and ${\rm sgn}(1-c^2) = -{\rm sgn}(\alpha)$, where ${\rm sgn}(\mathtt{a})$ is the well-known signal of a certain real number $\mathtt{a}$.

 Let us discuss some contributions regarding the obtaining of solutions for the model $(\ref{travSG33})$. Indeed, within an algebraic framework, the tanh method and other related techniques have been employed to analyze generalized versions of this equation. In \cite{Lu2019}, the authors obtained exact solitary-wave and elliptic-function solutions of the sinh-Gordon equation
 by employing a modified direct algebraic scheme. Further explicit solutions for a generalized version of $(\ref{KF2})$ are provided in \cite{wazwaz2006}. In \cite{article1512} (see also \cite{ali2024}) the authors employed an expansion method to numerically approximate traveling-wave solutions of both the double sinh--Gordon equation and its generalized variant, given respectively by
$$
u_{tt} -\mathtt{k}u_{xx} + 2\alpha \sinh(u) + \beta \sinh(2u) = 0,$$
and
$$
u_{tt} - \mathtt{k}u_{xx} + 2\alpha \sinh(nu) + \beta \sinh(2nu) = 0.
$$
In both cases $\alpha$ and $\beta$ are real parameters, $\mathtt{k}$ is a positive coefficient, and $n$ is a positive integer.  
   
Regarding the existence of periodic solutions and their orbital stability for equation $(\ref{KF2})$, we can mention a few works. We highlight at least two of them. In \cite{Natali2011}, the author established the orbital stability of odd periodic traveling wave solutions in the space $H^1_{per,m}\times L^2_{per}$, under specific choices of $c$, employing the abstract method developed in \cite{grillakis1}. For clarity, we denote
$$H^1_{per,m}= \left\{ f\in H^1_{per}; \ \displaystyle \int_{0}^{L} f(x)dx=0\right\}.$$ 
   
 In \cite{Deconinck}, the author  showed that the elliptic solutions of the sinh–Gordon equation in $(\ref{KF2'})$ are spectrally stable by exploiting the integrable structure of the equation. The author built a suitable Lyapunov functional obtained from higher-order conserved quantities and used it to prove that these solutions are orbitally stable under subharmonic perturbations of any period. It is important to mention that the integrability is not verified for the equation $(\ref{KF2})$ when $\alpha=-1$. 
 
Despite these developments, several analytical aspects of the hyperbolic sinh–Gordon model remain insufficiently understood. In this context, the goal of the present work is to advance in the qualitative study of the model in \eqref{KF2} for the case $\alpha = -1$, focusing on the existence of periodic traveling waves that solve equation $(\ref{travSG33})$, as well as on the spectral stability or instability of such waves.
 
 To begin with, we observe that equation \eqref{KF2} for $\alpha=-1$ has an abstract Hamiltonian system form
\begin{equation}\label{hamilt31}
U_t=J \mathcal{E}'(U),
\end{equation}   where $U=(u,v)=(u,u_t)$, $J$ is the skew-symmetric matrix
\begin{equation}\label{J}
J=\left(
\begin{array}{cc}
 0 & 1\\
-1 & 0\end{array}
\right),
\end{equation}and $\mathcal{E}'$ denotes the Fr\'echet derivative of the functional (energy) $\mathcal{E}:H^1_{per}\times L^2_{per} \rightarrow \mathbb{R}$ defined by  
\begin{equation}\label{E}
	\mathcal{E}(u, u_t)
=\displaystyle\frac{1}{2}\int_{0}^{L} \left[u_x^2+u_t^2-2(\cosh(u)-1)\right] dx,
\end{equation}
which is a conserved quantity. Additionally, \eqref{KF2} features another conserved quantity defined in $H^1_{per}\times L^2_{per}$, expressed as
\begin{equation}\label{F}
	\mathcal{F}(u,u_t)=\displaystyle\int_{0}^{L}u_x u_t  dx.
\end{equation}
 
 Let us rewrite equation $(\ref{travSG33})$ for $\alpha = -1$ in the more convenient form
\begin{equation}\label{KF3}
\omega \varphi'' + \sinh(\varphi) = 0,
\end{equation}
where $\omega = 1 - c^{2}$ is positive, so that $c \in (-1,1)$. In this setting, the solution of \eqref{KF3} can be expressed in terms of a Jacobi elliptic function as
\begin{equation}\label{Sol2}
\varphi(x) = 2\operatorname{arctanh}\biggl(k\sn\biggl(\frac{4K(k)}{L} \  x ; k\biggr)\biggr),
\end{equation}
where $\sn$ is the Jacobi elliptic snoidal function, $k \in (0,1)$ is the modulus of the elliptic function, and $K(k)=\displaystyle\int_{0}^{\frac{\pi}{2}}\frac{1}{\sqrt{1-k^2\sin^2(\theta)}}d\theta$ is the complete elliptic integral of the first kind (see \cite{byrd} for details on Jacobi elliptic functions and their properties).

By applying the implicit function theorem, we can deduce that, for a fixed period $L \in (0,2\pi)$, there exists a smooth curve
 $$\omega\in (0,1)\mapsto\varphi_{\omega}=\varphi\in H_{per,m}^n,$$ for all integers $n \geq 1$, consisting of solutions to equation $(\ref{KF3})$. Still concerning the determination of periodic solutions, we adopt another approach. For a fixed $\omega \in (0,1)$ and $L\in (0,2\pi\sqrt{\omega})$, we prove the existence of a weak solution $\varphi\in H_{per,m}^1$ for equation $(\ref{KF3})$ using the mountain pass theorem.
 
 Our main objective in this manuscript is to determine the spectral stability and instability of the smooth curve given by $(\ref{Sol2})$ using analytical methods. We work in the space $\mathbb{L}_{per,m}^2=L_{per,m}^2\times L_{per,m}^2$, which consists of $L$-periodic pairs $(u,v)\in\mathbb{L}_{per}^2= L_{per}^2\times L_{per}^2$ with zero mean. This choice is motivated by the fact that although $\varphi$ in $(\ref{Sol2})$ is an odd function, it is also a periodic function with zero mean. As far as we know, the space $\mathbb{L}_{per,m}^2$ provides a more suitable setting for studying spectral stability, allowing us to present a positive and definitive picture of this property. This contrasts with the partial results in \cite{Natali2011}, which established only orbital stability in the space $H^1_{per,m}\times L^2_{per}$, and with \cite{Deconinck}, which proved spectral stability by relying on the integrability of equation $(\ref{KF2'})$.
 
  To do so, we first consider some fundamental elements. As a starting point, we define the functional
$\mathcal{G}(u, v)=\mathcal{E}(u, v)-c\mathcal{F}(u, v).$ Clearly, any solution of \eqref{KF3} satisfies $\mathcal{G}'(\varphi,c\varphi')=0$, that is, $(\varphi,c\varphi')$ is a critical point of $\mathcal{G}$. Consider the linear projection operator  
\begin{equation}\displaystyle\label{opconstrained2}\mathcal{L}_{\Pi}= \mathcal{G}''(\varphi,c\varphi')+\displaystyle \left(
\displaystyle	\begin{array}{ccc}
		\frac{1}{L}\int_{0}^{L}\cosh(\varphi)\cdot  dx& & 0 \\\\
		0  & & 0
	\end{array}\right)=\displaystyle\mathcal{L}+\displaystyle \left(
\begin{array}{ccc}
\frac{1}{L}\int_{0}^{L}\cosh(\varphi)\cdot  dx & & 0 \\\\
0  & & 0
\end{array}\right),\end{equation} 

\noindent where $\mathcal{L}$ is defined by

	\begin{equation}\label{matrixop313}
		\displaystyle \mathcal{L}=\left(
		\begin{array}{ccc}
		-\partial_x^2-\cosh(\varphi) & &c\partial_x\\\\
		\ \ \ \ -c\partial_x & & 1
		\end{array}\right).
	\end{equation}
	
\noindent	Operator $\mathcal{L}_{\Pi}$ is defined on $  \mathbb{L}_{per,m}^2$ and it has domain $H_{per,m}^2\times H_{per,m}^1$. Using Morse index theory, we are enable to show that $\mathcal{L}_{\Pi}$ has only one negative simple eigenvalue and zero is a simple eigenvalue associated to the eigenfunction $(\varphi',c\varphi'')$. To this end, we first prove that the operator $\mathcal{L}$, defined on $\mathbb{L}^2_{per} $, with domain $ H^2_{per} \times H^1_{per} $, possesses two negative eigenvalues which are simple and zero is a simple eigenvalue associated with eigenfunction $(\varphi',c\varphi'')$, by using the monotonicity of the period mapping for the equation $(\ref{KF3})$.
	
	 In order to obtain  spectral stability, we need to calculate the Hamiltonian-Krein index given by analysing the difference
	\begin{equation}\operatorname{n}(\mathcal{L})-\operatorname{n}(\mathcal{D}),
	\label{Kham}\end{equation}
	where $\operatorname{n}(\mathcal{L})$ stands the number of negative eigenvalues (counting multiplicities) of the linearized operator $\mathcal{L}$ in $(\ref{matrixop313})$, and $\operatorname{n}(\mathcal{D})$ indicates the number of negative eigenvalues of the matrix
	\begin{equation}\label{matrixD''1}
 	\mathcal{D}=\left[\begin{array}{llllll} \displaystyle \int_{0}^{L}(\varphi')^2dx-2c^2\left(\frac{d\omega}{dk}\right)^{-1}\frac{\partial}{\partial k}\int_{0}^{L}(\varphi')^2dx
 		& &  0& & 	 	0\\\\0& &D_1& & 0\\\\
 		0& & 0& & L\end{array}\right],
 \end{equation}
 where $D_1=(\mathcal{L}_{1}^{-1}1,1)_{L_{per}^2}$ and $\mathcal{L}_1=-\omega\partial_x^2-\cosh(\varphi)$. It is important to mention that the $2\times 2$ sub-matrix in $\eqref{matrixD''1}$ given by $D=\displaystyle\left[\begin{array}{cc}D_1&  0\\
 				0&  L\end{array}\right]$ establishes the quantity  $\operatorname{n}(\mathcal{L}_{\Pi})$.
 				
  According to \cite[Theorem 7.1.5]{kapitula}, the periodic wave $(\varphi,c\varphi')$ is spectrally stable if the difference in (\ref{Kham}) is zero and spectrally unstable if the difference is odd. In our case, we prove for the explicit solution (\ref{Sol2}) that the pair $(\varphi,c\varphi')$ can be spectrally stable or unstable, depending on the value of the elliptic modulus $k\in (0,1)$. In summary, we have the following result:

\begin{theorem} \label{stabthm}
	Let $L\in(0,2\pi)$ be fixed and consider the periodic traveling wave solution in $(\ref{Sol2})$. There exists a unique $k_0=k_0(L) \in (0,1)$ such that:
	\begin{itemize}
		\item [(i)] if $k\in (0,k_0)$, the periodic wave $(\varphi,c\varphi')$ is spectrally stable in $\mathbb{L}_{per,m}^2$;
		\item [(ii)] if $k\in [k_0,1)$, the periodic wave $(\varphi,c\varphi')$ is spectrally unstable in $\mathbb{L}_{per,m}^2$.
	\end{itemize}
\end{theorem}
\begin{remark} In the first item of Theorem $\ref{stabthm}$, we can also conclude that the periodic traveling wave solution in $(\ref{Sol2})$ is orbitally stable in $H_{per,m}^1\times L_{per,m}^2$. In fact, in this case, we can deduce that there exists $C>0$ such that
\begin{equation}\label{ineqstab}
(\mathcal{L}(u,v),(u,v))_{\mathbb{L}_{per}^2}=(\mathcal{L}_{\Pi}(u,v),(u,v))_{\mathbb{L}_{per}^2}\geq C||(u,v)||_{\mathbb{L}_{per}^2}^2,
\end{equation}
for all $(u,v)\in H_{per,m}^2\times H_{per,m}^1$ such that $((u,v),(\varphi',c\varphi''))_{\mathbb{L}_{per}^2}=0$.
The estimate in $(\ref{ineqstab})$ then gives, by using classical approaches to orbital stability as in \cite{grillakis1}, that the periodic wave $(\varphi,c\varphi')$ is orbitally stable in the energy space $H_{per,m}^1\times L_{per,m}^2$. 
\end{remark}
\begin{remark} We are considering the study of spectral stability in the set $\mathbb{L}_{per,m}^2$ instead of the classical set $\mathbb{L}_{per}^2$. In addition, as we will see in Section \ref{section5}, we do not use the classical linearization of the Hamiltonian structure in \eqref{hamilt31} to obtain the spectral stability result stated in Theorem \ref{stabthm}. The reason is that the classical approaches for establishing spectral stability in $\mathbb{L}_{per}^2$ may fail for the sinh–Gordon equation. In fact, let us define $\displaystyle d(c)=\mathcal{E}(\varphi,c\varphi')-c\mathcal{F}(\varphi,c\varphi')$. We have that $(\varphi,c\varphi')$ is a critical point of the functional $\displaystyle\mathcal{G}(u,v)=\mathcal{E}(u,v)-c\mathcal{F}(u,v)$, so that by the chain rule
	 $$d''(c)=-\frac{d}{dc}\left(c\int_0^L(\varphi')^2dx\right)=-\int_0^L(\varphi')^2dx+2c^2\frac{\partial}{\partial\omega}\int_0^L(\varphi')^2dx.$$
	
\noindent	 According to \cite[Section 4]{Natali2011}, we have that $d''(c)<0$ on an open subset contained in $(-1,1)$. Since $\operatorname{n}(\mathcal{L})=2$, we obtain at these points
$$\operatorname{n}(\mathcal{L})-\operatorname{n}(-d''(c))=2-0=2,$$
an even number. By using the classical approaches for spectral stability (for instance, those based on \cite{grillakis1}), we see that the spectral stability in the whole space $\mathbb{L}_{per}^2$ is therefore inconclusive. Hence, our results provide a more complete framework for addressing these questions. Moreover, our new results appear promising and can be adapted to study other Klein–Gordon and wave-type equations in the periodic setting.
\end{remark}
\indent Our paper is organized as follows. Section \ref{section3}
discusses the well-posedness of the Cauchy problem associated with the evolution equation \eqref{KF2}, with $\alpha=-1$, in the space $H_{per,m}^3\times H_{per,m}^2$, including conditions leading to  blow-up phenomena. In Section \ref{section4}, we construct periodic traveling waves: weak solutions are obtained for fixed wave-speed and variable period via the mountain pass theorem, whereas explicit classical solutions arise in the complementary regime of fixed period and varying wave speed. Section \ref{spectralKG} is devoted to the spectral analysis of the linearized operators generated by the relative energy of the sinh-Gordon equation, where we characterize their kernels, spectra, and the number of negative eigenvalues. Section \ref{section5} develops the corresponding spectral stability and instability theory, relying on suitable transformations based on the previously derived periodic waves. \\

 \textbf{Notation.} 
 We begin by recalling the fundamental notation for periodic Sobolev spaces. For a more detailed treatment of these spaces, the reader is referred to \cite{Iorio}. By $L^2_{per}=L^2_{per}([0,L])$, $L>0$, we denote the space of all square integrable real functions which are $L$-periodic. For $s\geq0$, the Sobolev space
$H^s_{per}=H^s_{per}([0,L])$
is the set of all periodic real functions such that\noindent
$$
\|f\|^2_{H^s_{per}}= L \sum_{k=-\infty}^{+\infty}(1+|k|^2)^s|\hat{f}(k)|^2 <+\infty,
$$\noindent
where $\hat{f}$ is the periodic Fourier transform of $f$. The space $H^s_{per}$ is a  Hilbert space with natural inner product denoted by $(\cdot, \cdot)_{H^s_{per}}$. When $s=0$, the space $H^s_{per}$ is isometrically isomorphic to the space  $L^2_{per}$, that is, $L^2_{per}=H^0_{per}$ (see, e.g., \cite{Iorio}). The norm and inner product in $L^2_{per}$ will be denoted by $\|\cdot \|_{L^2_{per}}$ and $(\cdot, \cdot)_{L^2_{per}}$, respectively. 

\section{Well-posedness theory}\label{section3}

 This section is devoted to prove existence and uniqueness for the Cauchy problem associated with equation \eqref{KF2} considering $\alpha=-1$, and to analyse blow-up conditions for its solutions.
 
\subsection{Local well-posedness} 
Let us consider the classical Cauchy problem associated with the evolution equation $(\ref{KF2})$
\begin{equation}\label{CPKG}
	\left\{\begin{array}{llll}
		u_{tt}-u_{xx}-\sinh(u)=0, \:\:\:\: \text{in} \: [0,L]\times(0,+\infty), \\\\
		u(x,0)=u_0(x),\:\:\:\:\:\:\:\:\:\:\:\:\:\:\:\:\:\:\: \text{in} \: [0,L],\\\\
		u_t(x,0)= u_1(x),\:\:\:\:\:\:\:\:\:\:\:\:\:\:\:\:\:\: \text{in} \: [0,L].
	\end{array}\right.
\end{equation}
To study existence and uniqueness for \eqref{CPKG}, we start by rewriting it in matrix form as
\begin{equation}\label{CPKG1}
	\begin{cases}
		\left(\begin{array}{lll}u\\v\end{array}\right)_t
		=\left(\begin{array}{lll}0&  1\\ \partial_x^2&  0\end{array}\right)\left(\begin{array}{lll}u\\v\end{array}\right)+\left(\begin{array}{ccc}0\\ \sinh(u)\end{array}\right),\:\:\:\:\:\:\:\:\:\: \text{in} \: [0,L]\times(0,+\infty),\\\\
		\left(\begin{array}{llll}u(0)\\v(0)\end{array}\right)=\left(\begin{array}{ccc}u_0\\ v_0\end{array}\right),\:\:\:\:\:\:\:\:\:\:\:\:\:\:\:\:\:\:\:\:\:\:\:\:\:\:\:\:\:\:\:\:\:\:\:\:\:\:\:\:\:\:\:\:\:\:\:\:\:\:\:\: \:\:\:\:\text{in} \: [0,L],
	\end{cases}
\end{equation}
where $v=u_t$. We cannot garantee that the Cauchy problem \eqref{CPKG1} is locally well-posed in the Sobolev space $H_{per,m}^2\times H_{per,m}^1$ when employing the standard semigroup approach as in \cite{pazy}. In order to solve this issue, we must analyse the auxiliary Cauchy problem corresponding to the problem \eqref{CPKG}, which is expressed as
\begin{equation}\label{CPKG2}
	\begin{cases}
		\left(\begin{array}{lll}u\\w\end{array}\right)_t
		=\left(\begin{array}{lll}\partial_x^{-1}&  0\\ 0&  \partial_x\end{array}\right)\left(\begin{array}{lll}0&  1\\ \partial_x^2&  0\end{array}\right)\left(\begin{array}{lll}u\\w\end{array}\right)+\left(\begin{array}{ccc}0\\ \partial_x(\sinh(u))\end{array}\right),\:\:\:\:\:\text{in} \: [0,L]\times(0,+\infty),\\\\
		\left(\begin{array}{lll}u(0)\\w(0)\end{array}\right)=\left(\begin{array}{ccc}u_0\\ w_0\end{array}\right), \:\:\:\:\:\:\:\:\:\:\:\:\:\:\:\:\:\:\:\:\:\:\:\:\:\:\:\:\:\:\:\:\:\:\:\:\:\:\:\:\:\:\:\:\:\:\:\:\:\:\:\:\:\:\:\: \:\:\:\:\: \:\:\:\:\: \:\:\:\:\: \:\:\:\:\:\:\:\:\:\:\:\:\:\:  \text{in} \: [0,L],
	\end{cases}
\end{equation}
where $w=\partial_xu_t$ and $\partial_x^{-1}:L_{per,m}^2\rightarrow H_{per,m}^1$ is the well-known anti-derivative bounded linear operator.

Since the pair $(u,w)$ is a smooth solution to the equation in $(\ref{CPKG2})$ in an admissible space, such as $H_{per,m}^3\times H_{per,m}^1$, it follows that $u$ is a smooth solution of the projected Cauchy problem 
\begin{equation}\label{CPKG3}
	\left\{\begin{array}{llll}
		\displaystyle u_{tt}-u_{xx}-\sinh(u)+\frac{1}{L}\int_0^L\sinh(u)dx=0, \:\:\:\: \:\text{in} \: [0,L]\times(0,+\infty), \\\\
	u(x,0)=u_0(x),\:\:\:\:\:\:\:\:\:\:\:\:\:\:\:\:\:\:\:\:\:\:\:\:\:\:\:\:\:\:\:\:\:\:\:\:\:\:\:\:\:\:\:\:\:\:\:\:\:\:\:\:\:\:\:\: \text{in} \: [0,L],\\\\
		u_t(x,0)= u_1(x),\:\:\:\:\:\:\:\:\:\:\:\:\:\:\:\:\:\:\:\:\:\:\:\:\:\:\:\:\:\:\:\:\:\:\:\:\:\:\:\:\:\:\:\:\:\:\:\:\:\:\:\:\:\:\: \text{in}\ [0,L],
	\end{array}\right.
\end{equation}
with the zero-mean property.  

We start with the following elementary lemma:

\begin{lemma}\label{lemmaC0}
	The operator $\displaystyle A=\left(\begin{array}{lll}\partial_x^{-1}&  0\\ 0&  \partial_x\end{array}\right)\left(\begin{array}{lll}0&  1\\ \partial_x^2&  0\end{array}\right)$ defined on $X=H_{per,m}^2\times L_{per,m}^2$ with domain $D(A)=H_{per,m}^3\times H_{per,m}^1$ is a generator of a $C_0-$semigroup of 
	contractions $\{S(t)\}_{t\geq0}$ on the space $X$.
\end{lemma}
\begin{proof} See \cite[Lemma 2.1]{lonardoni}. 
\end{proof}

\indent In the next result, we demonstrate the existence of a local solution of the Cauchy problem $(\ref{CPKG2})$.

\begin{lemma}\label{lwpthm1}
	Let $L>0$ be fixed and consider   $(u_0,w_0)\in H_{per,m}^3\times H_{per,m}^1$. There exists $t_{\rm{max}}>0$ and a unique local (strong) solution   $(u,w)$ of the Cauchy problem $(\ref{CPKG2})$ satisfying $$(u,w)\in C([0,t_{\rm{max}}),H_{per,m}^3\times H_{per,m}^1)\cap C^1([0,t_{\rm{max}}),H_{per,m}^2\times L_{per,m}^2) .$$
\end{lemma}
\begin{proof}	
	According to Lemma $\ref{lemmaC0}$, the operator $A$ generates a $C_0-$semigroup of 
	contractions $\{S(t)\}_{t\geq0}$ on the space $X$. Let us consider $(u_0,w_0)\in H_{per,m}^3\times H_{per,m}^1=D(A)\subset X$. Our first objective is to show that there exists a maximal time $t_{\rm{max}}>0$ and a unique function 
	\begin{equation}\label{eq1.lemma}
		W=(u,w)\in C([0,t_{\rm{max}}),H_{per,m}^2([0,L])\times L_{per,m}^2([0,L])), 
	\end{equation}
	such that, for $t\in [0,t_{\rm{max}})$, $W(t)$ satisfies the integral equation
	\begin{equation}\label{eq2.lemma}
		W(t)=(u(\cdot, t), w(\cdot, t))=S(t)(u_0,w_0)+\int_{0}^{t}S(t-s)(0,\partial_x(\sinh( u(\cdot,s))))ds. 
	\end{equation}
	
	Initially, let us define the function $Q:H_{per,m}^2\times L_{per,m}^2\rightarrow H_{per,m}^2\times L_{per,m}^2$ given by $$Q(u,w)=(0, \partial_x(\sinh(u))).$$ 
	To verify that $Q$ is well-defined, it is enough to prove that $\partial_x(\sinh(u))\in L_{per,m}^2 $ for all $u \in H_{per,m}^2$. In fact, since $u\in H_{per}^1 \hookrightarrow  L_{per}^\infty$, then $|u|\leq \|u\|_{L^\infty_{per}}<+\infty$. Moreover, 
	$\partial_x u\in H_{per,m}^1\subset L_{per,m}^2$ and $e^{-2u}\leq e^{2|u|}$. Thus, we have 
	\begin{align*}
		\int_{0}^{L}|\partial_x(\sinh(u))|^2dx
		=&\frac{1}{4}\int_{0}^{L}(e^{2u}+2+e^{-2u})|\partial_xu|^2dx
	  \leq   \frac{1}{2}(1+e^{2\|u\|_{L^\infty_{per}}})\|\partial_xu\|_{L^2_{per}}^2 ,
	\end{align*} proving that  $\partial_x(\sinh(u))\in L^2([0,L]) $. In addition, the zero-mean property and the periodicity of $\partial_x(\sinh(u))$ it is a direct consequence of the regularity of the function $u\in L_{per,m}^2$. Therefore, $Q$ is well-defined. 
	
	Next, we proceed to demonstrate a fundamental property: let $R>0$ be fixed and suppose that 
	$(u_1, w_1), \  (u_2, w_2)\in X=H_{per,m}^2\times L_{per,m}^2$ satisfy 
	\begin{equation}\label{lim.eq7.1.lemma}  
		\| (u_1, w_1)\|_{X}, \| (u_2, w_2)\|_{X}\leq R.
	\end{equation} Then, there exists $M=M(L,R)>0$ such that
	\begin{equation}\label{eq7.1.lemma}  
		\|Q(u_1, w_1)-Q(u_2, w_2)\|_X\leq  M \|(u_1,w_1)-(u_2, w_2)\|_X.
	\end{equation}
	
	To begin with, we observe that 
	\begin{equation}\label{eq3.lemma}\begin{array}{llll}
			\|Q(u_1,w_1)-Q(u_2, w_2)\|_X &= & \| \cosh(u_1)\partial_xu_1-\cosh(u_2)\partial_xu_2 \|_{L_{per}^2}
			\\\\ &\leq & \| [\cosh(u_1)- \cosh(u_2)]\partial_xu_1\|_{L_{per}^2}\\\\
			&+&\|\cosh(u_2)[\partial_xu_1-\partial_xu_2 ] \|_{L_{per}^2}\\\\ &\leq & \| \cosh(u_1)- \cosh(u_2)\|_{L_{per}^2}\|\partial_xu_1\|_{L_{per}^\infty}\\\\ &+&\|\cosh(u_2)\|_{L_{per}^\infty}\|u_1-u_2  \|_{H_{per}^1}. 
	\end{array}\end{equation}
	From the mean value theorem, the hypothesis \eqref{lim.eq7.1.lemma}, and the continuity of the hyperbolic sine function, it follows that there exists a constant $M_1=M_1(L,R)>0$ satisfying 
	\begin{align}\label{eq4.lemma}
		\| \cosh(u_1)- \cosh(u_2)\|_{L_{per}^2}\leq M_1 \|u_1-u_2  \|_{L_{per}^2}.
	\end{align}
	Using the Sobolev embeddings $\displaystyle H_{per,m}^2 \hookrightarrow H_{per,m}^1\hookrightarrow L_{per,m}^\infty$ and $\displaystyle H_{per,m}^2\hookrightarrow L_{per,m}^2$, together the inequalities \eqref{eq3.lemma} and \eqref{eq4.lemma}, we obtain constants  $M_i=M_i(L,R)>0$, $i=2,3$, such that 
	\begin{align}\begin{array}{llll}\label{eq5.lemma}
		\|Q(u_1,w_1)-Q(u_2, w_2)\|_X &\leq &  M_1 \|u_1-u_2  \|_{L_{per}^2}\|\partial_xu_1\|_{L_{per}^\infty} +\|\cosh(u_2)\|_{L_{per}^\infty}\|u_1-u_2  \|_{H_{per}^1} 
		\\\\ &\leq & M_2\|u_1-u_2  \|_{H_{per}^2} \|\partial_xu_1\|_{H_{per}^1} +M_3\|u_1-u_2  \|_{H_{per}^2} \\\\ &\leq  &( M_2\|u_1\|_{H_{per}^2} +M_3)\|(u_1,w_1)-(u_2, w_2)\|_X.\end{array}
	\end{align}From \eqref{eq5.lemma} and the assumption $\|u_1\|_{H_{per}^2}\leq \| (u_1, w_1)\|_{X}\leq R$, we conclude that there exists a constant $M=M(L,R)>0$ such that \eqref{eq7.1.lemma} holds. 
	
	For a given $T>0$, let us define the set
	\begin{equation}\label{eq8.lemma}  
		\Upsilon=\left\{(u,w)\in C([0,T], X);\ \sup_{t\in [0,T]}\|(u(\cdot,t), w(\cdot,t))\|_X\leq 1+\|(u_0, w_0)\|_X\right\}.
	\end{equation}
	Since the set $	\Upsilon$ is closed in $C([0,T],X)$ with the supremum norm, it is a complete metric space. We also introduce the mapping $\Phi: \Upsilon\rightarrow\Upsilon$ defined, for each $t\in [0,T]$, by
	\begin{equation}\label{eq9.lemma}  
		\Phi(u(\cdot,t), w(\cdot, t))=S(t)(u_0,w_0)+\int_{0}^{t}S(t-s)Q(u(\cdot,s),w(\cdot,s))ds.
	\end{equation}
	
	With the aim of proving existence and uniqueness for the abstract Cauchy problem \eqref{CPKG2} via Banach’s Fixed Point Theorem, we show that the mapping $\Phi$ is well defined on an open ball of radius $r>0$, and that it is a strict contraction. To this end, let us take $r = 1 + \|(u_0, w_0)\|_X > 0$ and fix an arbitrary element $(u,w) \in \Upsilon$.  
	It is worth noting that the parameter $r$ depends on the time variable $t$ because of the restriction imposed in the definition of the set $\Upsilon$. Next, an appropriate choice of $T>0$ is required to ensure the well-definedness of $\Phi$. Indeed, from \eqref{eq7.1.lemma} we obtain that there exists a positive constant $M=M(L,r)$, which now depends implicitly on the maximal time $T$ due to the choice of $r$, such that
	\begin{equation}\label{eq10.lemma}  
		\|Q(u(\cdot,t), w(\cdot,t))-Q(u_0, w_0)\|_X\leq  M \|(u(\cdot,t),w(\cdot,t))-(u_0, w_0)\|_X,
	\end{equation}for all $t\in [0,T]$ and $(u,w)\in\Upsilon$. Furthermore, from equation \eqref{eq10.lemma} and the condition that $S(t)$ is a $C_0$-semigroup of contractions, we obtain 
	\begin{equation}\label{eq11.lemma}  
		\|\Phi(u(\cdot,t), w(\cdot, t))\|_X\leq \|(u_0,w_0)\|_X+ MT[1+2\|(u_0,w_0)\|_X]  +T\|(0,\partial_x(\sinh(u_0)))\|_X. 
	\end{equation}
	Considering 
	\begin{equation}\label{eq13.lemma}  
		0<T^*=\left\{  M[1+2\|(u_0,w_0)\|_X]+\|(0,\partial_x(\sinh(u_0)))\|_X \right\}^{-1}<+\infty, 
	\end{equation}
	and redefining $T$ so that $0 < T \le T^{*}$, it follows from 
	\eqref{eq11.lemma} and \eqref{eq13.lemma} that for all $t \in [0,T]$ we have
	\begin{equation}\label{eq14.lemma}  
		\|\Phi(u(\cdot,t), w(\cdot, t))\|_X\leq1+\|(u_0,w_0)\|_X=r, 
	\end{equation}which establishes that the mapping $\Phi$ is well defined.
	
	\indent  We now proceed to show that the function satisfies a strict contraction property. To this end, consider $(u_1,w_1), (u_2,w_2)\in \Upsilon$. Proceeding as in \eqref{eq7.1.lemma} and \eqref{eq10.lemma}, 
	we deduce that
	\begin{equation}\label{eq16.lemma}  
		\|Q(u_1, w_1)-Q(u_2, w_2)\|_{C([0,T],X)}\leq  M T\|(u_1,w_1)-(u_2, w_2)\|_{C([0,T],X)}.
	\end{equation} 
	Since $T \le T^{*}$ and, by \eqref{eq13.lemma}, we have $T^{*} < \frac{1}{M}$, it follows from \eqref{eq16.lemma} that $\Phi$ is a strict contraction. In view of these conditions, Banach’s Fixed Point Theorem ensures the existence and uniqueness of a function $(u, w)\in \Upsilon$ such that $\Phi( u, w)=(u, w),$ that is, \eqref{eq2.lemma} holds. 
	
	From this point on, let us set $t_{\max} = T^{*}$. By applying Gronwall’s inequality, and using the fact that $Q$ satisfies \eqref{eq7.1.lemma}, one shows that the function $(u, w) \in C([0,t_{\max}), X)$ obtained above is the unique mild solution of the Cauchy problem on the interval $[0, t_{\max})$.
	
	Our next goal is to show that $(u,w)\in C([0,t_{\rm{max}}),D(A))\cap C^1([0,t_{\rm{max}}),X) $. To proceed, consider $W_0=(u_0, w_0)$, $W(t)= (u(\cdot, t), w(\cdot, t))$ and $ z(t)=\int_{0}^{t}S(t-s)Q(W(s))ds$. Since
	\begin{equation}\label{eq17.lemma}  
		W(t)= S(t) W_0+z(t),
	\end{equation}we obtain from
	\cite[Chapter 1, Theorem 2.4]{pazy} that $z(t)\in D(A)$, $S(t)W_0 \in D(A)$, and 
	\begin{equation}\label{eq18.lemma}  
		\frac{d}{dt}S(t) W_0=AS(t)W_0=S(t)AW_0.
	\end{equation}
	In addition, using the property in \eqref{eq7.1.lemma} for the function $Q$, together with the fact that $S(t)$ is a $C_0$–semigroup and the continuity $W\in C([0,t_{\rm max}),X)$ obtained from the fixed-point argument, we prove that $z$ is differentiable and satisfies
	\begin{equation}\label{eq19.lemma}  
		\frac{d}{dt}z(t)=A(z(t))+Q(W(t)).
	\end{equation}
	From \eqref{eq18.lemma} and \eqref{eq19.lemma} we obtain that \begin{equation}\label{eq20.lemma}  
		\frac{d}{dt}W(t)= AW(t)+Q(W(t)) . 
	\end{equation} Therefore, $ W\in C([0,t_{\rm{max}}),D(A))\cap C^1([0,t_{\rm{max}}),X)$ is the unique local (strong) solution of the Cauchy problem $(\ref{CPKG2})$, where $t_{max}=T^*$ depends on the norm of the initial data $(u_0,w_0)\in D(A)$ and on the fixed period $L>0$.	
\end{proof}

The arguments presented in Lemma \ref{lwpthm1} can be used to establish a local well-posedness result for the Cauchy problem in \eqref{CPKG3}.

\begin{proposition}[Local well-posedness for the projected Cauchy problem]\label{lwpthm}
	Let $L>0$ be fixed and consider   $(u_0,u_1)\in H_{per,m}^3\times H_{per,m}^2$. There exists $t_{\rm{max}}>0$ and a unique  local (strong) solution $u$ of the Cauchy problem $(\ref{CPKG3})$ satisfying $$(u,u_t)\in C([0,t_{\rm{max}}),H_{per,m}^3\times H_{per,m}^2)\cap C^1([0,t_{\rm{max}}),H_{per,m}^2\times L_{per,m}^2) .$$ 
\end{proposition}
\begin{proof}
	Let $(u_0,u_1)\in H_{per,m}^3\times H_{per,m}^2$. Hence, setting $(u_0,w_0)=(u_0,\partial_x u_1)$, we have $(u_0,w_0)\in D(A)$, and by Lemma \ref{lwpthm1} there exist $t_{\rm{max}}>0$ and a unique strong solution\noindent $$(u,w)\in C([0,t_{\rm{max}}),H_{per,m}^3\times H_{per,m}^1)\cap C^1([0,t_{\rm{max}}),H_{per,m}^2\times L_{per,m}^2)$$ of the Cauchy problem 
	$(\ref{CPKG2})$. Since $(u,w)$ is a strong solution of \eqref{CPKG2}, it follows that the pair $(u,w)$ satisfies the following system of partial differential equations	
	\begin{equation}\label{sysKG}
		\left\{\begin{array}{llll}u_t=\partial_x^{-1}w,\\\\
			w_t=\partial_x(u_{xx}+\sinh(u)).
		\end{array}\right.
	\end{equation}	
Taking the time derivative of the first equation in \eqref{sysKG}, applying the operator $\partial_x^{-1}$ to the second equation, and comparing the results, we obtain that $u$ satisfies the first equation in \eqref{CPKG3}.

	Moreover, from $u_t=\partial_x^{-1}w$, we deduce that  $(u_0,w_0)
	=(u_0,\partial_xu_t(0))=(u_0,\partial_xu_1)$. This concludes that $u$ is the unique strong solution to the Cauchy problem $(\ref{CPKG3})$ satisfying  $(u_0,\partial_x^{-1}w_0)=(u_0,u_1)=(u(0),u_t(0))$, and $$(u,u_t)\in C([0,t_{\rm{max}}),H_{per,m}^3\times H_{per,m}^2)\cap C^1([0,t_{\rm{max}}),H_{per,m}^2\times L_{per,m}^2).$$ 
\end{proof}

\begin{remark}\label{cqEF} 
The two basic conserved quantities in $(\ref{E})$ and $(\ref{F})$ associated with the problem $(\ref{KF2})$ can be deduced. Indeed, since $(u, u_t)\in C^1([0,t_{\rm{max}}),H_{per,m}^2\times L_{per,m}^2)$, it follows from Lemma \ref{lwpthm1}, after multiplying the equation in $(\ref{CPKG3})$ by $u_t$, integrating the final result over $[0,L]$ and using the fact that $u_t$ has the zero mean property, we obtain
	\begin{equation}\label{ded2}
		\frac{1}{2}\frac{d}{dt}	 \int_{0}^{L} \left[u_x^2+u_t^2-2(\cosh(u)-1)\right] dx =0. 
	\end{equation}
	Therefore, by \eqref{ded2} we have the conserved quantity in $(\ref{E})$. \\
	\indent We show that $\mathcal{F}(u,u_t)=\displaystyle\int_{0}^{L}u_xu_t dx$ is also a conserved quantity. In fact, by the equation in $\eqref{CPKG3}$, we obtain that
	\begin{equation}\label{ded3}\begin{array}{llll}
			\displaystyle\frac{d}{dt}\mathcal{F}(u,u_t)&=& \displaystyle\int_{0}^{L}(u_tu_{tx}+ u_x u_{tt}) dx \\\\ &= &\displaystyle\int_{0}^{L}\left(u_t u_{tx}+ u_x u_{xx}+u_x\sinh(u)-u_x\frac{1}{L}\int_{0}^{L}\sinh(u)dx\right) dx \\\\ &= &\displaystyle\int_{0}^{L}\left(\frac{1}{2}\frac{d}{dx} u_t ^2
			+ \frac{1}{2}\frac{d}{dx} u_x ^2+\frac{d}{dx} \cosh(u)-u_x\frac{1}{L}\int_{0}^{L}\sinh(u)dx\right)dx.
	\end{array}\end{equation}
	As a consequence of the periodicity of the functions $u$, $u_x$, and $u_t$, we conclude from \eqref{ded3} that $\frac{d}{dt}\mathcal{F}(u,u_t)=0$, so that $\mathcal{F}$ is also a conserved quantity.
\end{remark}

Given the positivity of the hyperbolic cosine function and its exponential growth, it is not possible to establish the existence of global weak solutions to the Cauchy problem \eqref{CPKG3}, nor for problem \eqref{CPKG}. Instead, one can only guarantee the existence of a local weak solution, whose statement follows. 

\begin{proposition}[Existence of a local weak solution for the classical Cauchy problem]\label{lwpthm2}	
	Let $L>0$ be fixed and consider $(u_0,u_1)\in Y=H_{per,m}^1\times L_{per,m}^2$. There exists a unique local weak solution $u$ in $H_{per,m}^1$ of the Cauchy problem $(\ref{CPKG})$ in the sense that 
	\begin{equation}\label{weaksol}
	\displaystyle	\langle u_{tt}(\cdot,t), \mathfrak{p}	\displaystyle\rangle_{H^{-1}_{per,m},H_{per,m}^1} + \int_{0}^{L} u_x(\cdot,t)\mathfrak{p}_xdx - \int_{0}^{L} \sinh(u(\cdot,t))\mathfrak{p}dx = 0, \ \text{a.e.} \ t \in [0,t_{max}),
	\end{equation}for all $\mathfrak{p}\in H_{per,m}^1$. Moreover, we have that $(u,u_t)\in C([0,t_{max}),Y)$.
\end{proposition}
\begin{proof} Let $(u_0,u_1)\in H_{per,m}^1\times L_{per,m}^2$. By density, there exists a sequence $\{ (u_{0,n},u_{1,n})\}_{n\in\mathbb{N}} \subset H_{per,m}^3\times H_{per,m}^2$ such that
$(u_{0,n},u_{1,n})$ converges to $ (u_0,u_1)  $ in $Y$. Under regular initial conditions for $(u_{0,n},u_{1,n})$ as stated in Proposition $\ref{lwpthm}$, we have the associated sequence of solutions 
\begin{align} \label{teo.fraco.eq1}
	(u_n, u_{t,n})\in  C([0,t_{\rm{max,n}}),H_{per,m}^3\times H_{per,m}^2)\cap C^1([0,t_{\rm{max,n}}),H_{per,m}^2\times L_{per,m}^2),
\end{align}so that 
\begin{align} \label{teo.fraco.eq2}
	u_{tt,n}-u_{xx,n}-\sinh (u_n)+\frac{1}{L}\int_0^L\sinh(u_n)dx=0,  \quad  {\rm{in}} \: [0,L]\times(0,t_{\rm{max,n}}).
\end{align}
To simplify the notation, let us denote $\phi=u_{\tilde{n}}$, $\eta=u_n$, $\phi_t=u_{t,{\tilde{n}}}$, and $\eta_t=u_{t,n}$. Given these assumptions, and defining $\xi=\phi-\eta$, we have that
\begin{align} \label{teo.fraco.eq4}
	\xi_{tt}-\xi_{xx}-[\sinh(\phi)-\sinh(\eta)]+\frac{1}{L}\int_0^L\left[\sinh(\phi)-\sinh(\eta)\right]dx=0,\  {\rm in}\ [0,L]\times  (0,t_{\rm{max, n}}),
\end{align}where $t_{\rm{max,n}}$ denotes, from now on, the minimum between the values $t_{\rm{max,{\tilde{n}}}}$ and $ t_{\rm{max,n}}$. Multiplying \eqref{teo.fraco.eq4} by $\xi_t$ and integrating in $x$ over $[0,L]$, it follows, since $\xi_t$ has the zero-mean property, that
\begin{align} \label{teo.fraco.eq5}
\frac{1}{2}\frac{d}{dt}	\displaystyle\int_0^L(\xi_t^2+\xi_x^2)dx\leq \int_0^L|\xi_t||\sinh(\phi)-\sinh(\eta)|dx.
\end{align}
 Moreover, using Young inequality into \eqref{teo.fraco.eq5} it follows that 
 \begin{align} \label{teo.fraco.eq5'}
 	\frac{1}{2}\frac{d}{dt}	\displaystyle\int_0^L(\xi_t^2+\xi_x^2)dx\leq \int_0^L\frac{|\xi_t|^2}{2}dx+\int_0^L\frac{|\sinh(\phi)-\sinh(\eta)|^2}{2}dx.
\end{align}

Next, integrating the result in \eqref{teo.fraco.eq5'} over the interval $[0,t]\subset[0,t_{\rm{max,n}})$, we obtain
\begin{equation} \label{teo.fraco.eq6}\begin{array}{lllll}
\displaystyle	\frac{1}{2} 	\displaystyle\int_0^L(\xi_t^2+\xi_x^2)dx&\leq \displaystyle  \frac{1}{2} 	\displaystyle\int_0^L(\xi_{t,0}^2+\xi_{x,0}^2)dx+\displaystyle	\frac{1}{2}\int_0^t\int_0^L\xi_t^2 dxdt\\ \\ & \ \ +\displaystyle	\frac{1}{2}\int_0^t\int_0^L|\sinh(\phi)-\sinh(\eta)|^2dxdt. 
\end{array}
\end{equation}
Using the relation $\sinh(\phi)-\sinh(\eta)=2\cosh(\frac{\phi+\eta}{2} )\sinh(\frac{\phi-\eta}{2} )$ and, for $x\ge 0$, the inequality $0\leq \sinh(x)\leq x\cosh(x)$, it follows that 
\begin{equation}\label{teo.fraco.eq7}\begin{array}{lllll}
\displaystyle	\int_0^L|\sinh(\phi)-\sinh(\eta)|^2dx&=&\displaystyle	\int_0^L\biggl| 2\cosh\biggl(\frac{\phi+\eta}{2} \biggl)\sinh\biggl(\frac{\phi-\eta}{2} \biggl)\biggl| ^2dx \\\\ &= &\displaystyle 4	\int_0^L\biggl| \cosh\biggl(\frac{\phi+\eta}{2} \biggl)\biggl| ^2\biggl|\sinh\biggl(\frac{\xi}{2} \biggl)\biggl| ^2dx \\\\ &= &\displaystyle 4	\int_0^L\biggl| \cosh\biggl(\frac{\phi+\eta}{2} \biggl)\biggl| ^2\biggl[\sinh\biggl(\biggl|\frac{\xi}{2} \biggl|\biggl)\biggl] ^2dx\\\\ &\leq &\displaystyle  4	\int_0^L\biggl| \cosh\biggl(\frac{\phi+\eta}{2} \biggl)\biggl| ^2\biggl| \cosh\biggl(\biggl|\frac{\xi}{2} \biggl| \biggl)\biggl| ^2  \biggl|\frac{\xi}{2} \biggl|^2dx. 
\end{array}\end{equation}
 
 According to the construction in the proof of Lemma \ref{lwpthm1}, the maximal value $t_{max,n}$ satisfies the bound $t_{max,n}\leq T_n^*$ for each $n\in \mathbb{N}$, where $T_n^*$ is defined by\noindent
	\begin{equation}\label{eq13.lemma*}  
		0<T^*_n=\left\{  M_n[1+2\|(u_{0,n},w_{0,n})\|_X]+\|(0,\partial_x(\sinh(u_{0,n})))\|_X \right\}^{-1}<+\infty,
	\end{equation} with $M_n$ being the parameter arising from the condition as in \eqref{eq16.lemma} that depends on $n\in\mathbb{N}$. It is important to mention that $M_n$ depends on the period $L>0$ and $r_n=1+||(u_{0,n},w_{0,n})||_X$, where $w_{0,n}=\partial_x u_{1,n}\in H_{per,m}^1$. Since, in particular $\{(u_{0,n},w_{0,n})\}_{n\in\mathbb{N}}$ is a bounded sequence, we obtain, up to a subsequence, that $\{M_n\}_{n\in\mathbb{N}}$ converges to a non-negative number $M$. 
	
	 Under all these facts, we can see that $\{T^*_n\}_{n\in\mathbb{N}}$ defines a bounded sequence that depends on the norm of the initial data $(u_{0,n},u_{1,n}) \in H_{per,m}^3\times H_{per,m}^2$ and on the fixed period $L>0$. Thus, the sequence $\{t_{max,n}\}_{n\in \mathbb{N}}$ is bounded, which implies the existence of a convergent subsequence, which we will still denote in the same way, converging to $t_{max}$. 
 
Since the functions $\phi$, $\eta$, and $\xi$ are in $H_{per,m}^1\subset L_{per,m}^{\infty}$, we further observe that, for each $t\in[0, t_{\rm{max,n}})$, the following inequality holds:
 \begin{equation}\label{teo.fraco.eq07}\begin{array}{lllll}
\displaystyle\biggl|\cosh\biggl(\frac{\phi+\eta}{2} \biggl)\biggl| ^2\biggl| \cosh\biggl(\biggl|\frac{\xi}{2} \biggl| \biggl)\biggl| ^2 &\leq& \displaystyle\frac{(e^{\frac{\phi+\eta}{2}}+e^{-\frac{\phi+\eta}{2}})^2}{4} \frac{(e^{|\frac{\xi}{2} |}+e^{-|\frac{\xi}{2} |  })^2}{4}
\\\\ &\leq &  \displaystyle\frac{(e^{\frac{\|\phi\|_{L^\infty_{per}}+\|\eta\|_{L^\infty_{per}}}{2}}+e^{\frac{\|\phi\|_{L^\infty_{per}}+\|\eta\|_{L^\infty_{per}}}{2}})^2}{4} \frac{(e^{\frac{\|\xi\|_{L^\infty_{per}}}{2} }+e^{\frac{\|\xi\|_{L^\infty_{per}}}{2}  })^2}{4}
\\\\ &\leq &\displaystyle e^{(\|\phi\|_{L^\infty_{per}}+\|\eta\|_{L^\infty_{per}})}  e^{\|\xi\|_{L^\infty_{per}}}. 
 \end{array}\end{equation}

By applying the embedding  $H_{per,m}^1\hookrightarrow L_{per,m}^{\infty}$ into \eqref{teo.fraco.eq07}, we guarantee the existence of a constant $\tilde{M}_{0,n}= \tilde{M}_{0,n}(L, t_{max,n})>0$, which also depends on the norm of the initial data $(u_{0,n},u_{1,n})$, such that 
\begin{align}\label{teo.fraco.eq8}
	\int_0^L|\sinh(\phi)-\sinh(\eta)|^2dx&	\leq \tilde{M}_{0,n}
	\int_0^L\xi^2dx\leq \tilde{M}_{0,n} \biggl(\frac{L}{2\pi}\biggl)^2\int_0^L\xi_x^2dx, 
\end{align}where in the last inequality in \eqref{teo.fraco.eq8} we used the Poincaré–Wirtinger inequality in $H_{per,m}^1$.

From \eqref{teo.fraco.eq6} and \eqref{teo.fraco.eq8}, we obtain that

\begin{equation} \label{teo.fraco.eq9}\begin{array}{lllll}
		\displaystyle	\displaystyle\int_0^L(\xi_t^2+\xi_x^2)dx&\leq \displaystyle 	\displaystyle\int_0^L(\xi_{t,0}^2+\xi_{x,0}^2)dx+\displaystyle	\int_0^t\int_0^L\xi_t^2 dxdt +\displaystyle	\tilde{M}_{0,n} \biggl(\frac{L}{2\pi}\biggl)^2\int_0^t\int_0^L\xi_x^2dxdt. 
	\end{array}
\end{equation}Considering $\tilde{M}_n=\tilde{M}_n(L,t_{max,n})= \max \biggl\{ 1, 	\tilde{M}_{0,n} \biggl(\frac{L}{2\pi}\biggl)^2 \biggl\}$, we deduce from \eqref{teo.fraco.eq9} that 
\begin{equation} \label{teo.fraco.eq10}\begin{array}{lllll}
		\displaystyle	\displaystyle\int_0^L(\xi_t^2+\xi_x^2)dx&\leq \displaystyle 	\displaystyle\int_0^L(\xi_{t,0}^2+\xi_{x,0}^2)dx+\displaystyle	\tilde{M}_n\int_0^t\int_0^L(\xi_t^2 +\xi_x^2 )dxdt  . 
	\end{array}
\end{equation}
Applying Gronwall's inequality to \eqref{teo.fraco.eq10}, we conclude
\begin{equation} \label{teo.fraco.eq11} 
	 	\displaystyle\int_0^L(\xi_t^2+\xi_x^2)dx \leq \biggl[\int_0^L(\xi_{t,0}^2+\xi_{x,0}^2)dx\biggl ] e^{\tilde{M}_n t} ,
\end{equation}
where $t\in  [0,t_{max })$, with $t_{max}$ being the limit of the subsequence of times $\{t_{max,n}\}_{n\in \mathbb{N}}$.

It should be remarked that the inequality in \eqref{teo.fraco.eq11} does not allow for the extension to a global solution, since we have no control over the parameter $M$ due to its dependence on the maximal time. This implies that the product $Mt$ may fail to remain bounded for $t \in [0,t_{\max})$..

  Since $\xi=u_{\tilde{n}}-u_n$, \eqref{teo.fraco.eq11} shows that $	(u_n, u_{t,n})$ is a Cauchy sequence in $L^\infty([0,t_{max}),Y)$. Therefore, there exists $(u, u_t) \in L^\infty([0,t_{max}),Y)$, such that 
\begin{equation}  \label{teo.fraco.eq12}
	(u_n, u_{t,n})\rightarrow (u, u_t) \quad {\rm{in}} \quad L^\infty([0,t_{max}),Y).  
\end{equation}

Defining $U_0=(u_0,u_1)$, $U=(u, u_{t})$, $U_{n}=(u_n, u_{t,n})$, $U_{0,n}=(u_{0,n}, u_{1,n})$, and employing the time continuity of the conserved quantity  $\mathcal{E}$ in $(\ref{E})$, together with the convergence in \eqref{teo.fraco.eq12}, we obtain
\begin{align}\label{teo.fraco.eq13}
\begin{array}{ccc}
	\mathcal{E}(U_n(t)) & = & \mathcal{E}(U_{0,n}) \\
	\mathrel{\downarrow} & & \mathrel{\downarrow} \\
	\mathcal{E}(U(t)) & = & \mathcal{E}(U_0)
\end{array}
\end{align}

Using the convergences in \eqref{teo.fraco.eq13} and the arguments in \cite[Lemma 2.4.4]{cazenave}, we establish that $	(u, u_{t})\in  C([0,T],Y)$ for all $T\in [0,t_{max})$. Furthermore, by standard arguments of passage to the limit in \eqref{teo.fraco.eq2}, one can show that $(u, u_t)$ is a local weak solution of $\eqref{CPKG}$ in $C([0,T],Y)$ with initial data $(u_0,u_1)\in Y$ provided that $L>0$. Inequality \eqref{teo.fraco.eq11}, together with the identity $\xi=\phi-\eta$, also ensures the uniqueness of the weak solution in $C([0,T], Y)$.  Consequently, we conclude the existence and uniqueness of the local weak solution $(u, u_t) \in C([0,t_{max}), Y)$, as assert in Proposition $\ref{lwpthm2}$. 
\end{proof}

\begin{remark}\label{remweakstrong} The strong solution obtained in Proposition $\ref{lwpthm}$ can also be interpreted as a weak solution to $(\ref{CPKG})$ in $H_{per,m}^1$ that satisfies $(\ref{weaksol})$. Following the approach of Proposition $\ref{lwpthm2}$, we multiply the equation in $(\ref{CPKG3})$ by a function $\mathfrak{p} \in H_{per,m}^1$ and integrate  over the interval $[0,L]$. Thus, for $\alpha=-1$, the strong solution satisfies the weak formulation in $(\ref{weaksol})$. Consequently, the Cauchy problem $(\ref{CPKG})$ is solvable in the weak sense for strong solutions.
\end{remark}

\subsection{A blow-up analysis}

\begin{proposition}\label{bw}
Let $(u_0,u_1)\in H_{per,m}^3\times H_{per,m}^2$ and $u$ the local (strong) solution of the Cauchy problem \eqref{CPKG3}. If $\mathcal{E}(u_0,u_1)<0$, we have that $$\sup_{t>0}||u(t)||_{L_{per}^2}^2=+\infty.$$
\end{proposition}
\begin{proof}
	Indeed, according to Proposition \ref{lwpthm}, the local solution $u$ of the Cauchy problem \eqref{CPKG3} satisfies, in particular, $(u,u_t)\in C^1([0,t_{max}),H_{per,m}^2\times H_{per,m}^1)$. Furthermore, by Remark \eqref{cqEF} we obtain that $\mathcal{E}(u(t),u_t(t))= \mathcal{E}(t)=\mathcal{E}(0)=\mathcal{E}(u_0,u_1)$ is a conserved quantity. Then, we can write 
	\begin{equation}\label{blowup1}
		-2\int_0^Lu_x(x,t)^2dx=-4\mathcal{E}(0)+2\int_0^Lu_t(x,t)^2dx-4\int_0^L(\cosh(u(x,t))-1)dx.
	\end{equation}

	Now, let us consider the square norm function
	$\displaystyle \Lambda(t)=\int_0^Lu(x,t)^2dx=||u(t)||_{L_{per}^2}^2$. Differentiating  once function $\Lambda(t)$ with respect to $t>0$, we obtain that $\displaystyle \Lambda'(t)=2\int_0^Lu(x,t)u_t(x,t)dx$. Thus, after differentiating once more, we obtain, from the equation in $(\ref{CPKG3})$ that
	\begin{equation}\label{blowup2}\begin{array}{llllll}
		\Lambda''(t)&=&\displaystyle 2\int_0^L u_t(x,t)^2+2\int_0^Lu(x,t)u_{tt}(x,t)dx\\\\
		&=&\displaystyle  2\int_0^L u_t(x,t)^2+2\int_0^Lu(x,t)u_{xx}(x,t)dx+2\int_0^Lu(x,t)\sinh (u(x,t))dx\\\\
		&=&\displaystyle  2\int_0^L u_t(x,t)^2-2\int_0^Lu_{x}(x,t)^2dx+2\int_0^Lu(x,t)\sinh (u(x,t))dx. 
	\end{array}\end{equation}
	From \eqref{blowup1} and  \eqref{blowup2} it follows that 	 	
	\begin{align*} 
		\Lambda''(t)=&4\int_0^L u_t(x,t)^2-4\mathcal{E}(0)+2\int_0^Lu(x,t)\sinh (u(x,t))dx-4\int_0^L(\cosh(u(x,t))-1)dx.
	\end{align*}
	By using the basic inequality $2x\sinh(x)\ge 4(\cosh(x)-1)$ for all $x\in \mathbb{R}$, and the fact that $\mathcal{E}(0)<0$, we obtain 
	\begin{equation}\label{blowup3}\begin{array}{lllll}
		\Lambda''(t)\Lambda(t)&\ge& \displaystyle 4\int_0^Lu(x,t)^2dx\int_0^L u_t(x,t)^2dx -4\mathcal{E}(0) \int_0^Lu(x,t)^2dx \\\\
		 &\ge &\displaystyle  4 \biggl (\int_0^Lu(x,t)u_t(x,t)dx\biggl)^2 -4\mathcal{E}(0) \int_0^Lu(x,t)^2dx\\\\ 
		 &= & \displaystyle (	\Lambda'(t))^2-4\mathcal{E}(0) \Lambda(t).
		 \end{array}
	 \end{equation}
		Therefore, by assuming $u\notequiv 0$, we have that $\Lambda$ satisfies the inequality
		\begin{equation}\label{lambda1}
			\frac{\Lambda''(t)\Lambda(t)-(\Lambda'(t))^2}{\Lambda(t)}\ge -4\mathcal{E}(0). 
		\end{equation}
		Consider $-4\mathcal{E}(0)=d_0>0$. For $\displaystyle \Xi(t)=\ln(\Lambda(t))$, we have $\Xi'(t)= \frac{\Lambda'(t) }{\Lambda(t)}$ and $\Xi''(t)=	 \frac{	\Lambda''(t)\Lambda(t)-(	\Lambda'(t))^2}{(\Lambda(t))^2}$. Thus, it follows from $(\ref{lambda1})$ that $\Xi''(t)\ge \frac{ d_0}{\Lambda(t)}$.\\
		\indent Let us determine that $\displaystyle\sup_{t>0} \Lambda (t)=+\infty$. In fact, if $\Lambda(t) \leq \Lambda_0$ for all $t\ge0$ and for some $\Lambda_0>0$, then we have $\Xi''(t)\ge \frac{ d_0}{\Lambda_0}=c_0>0$. Integrating this inequality twice with respect to $t$, we obtain
		$		\Xi(t)\ge c_0t^2+b_0t+a_0,$ where $a_0,b_0\in \mathbb{R}.
		$	
		We then conclude that
			$\Lambda(t)= e^{\Xi(t)}\ge e^{ c_0t^2+b_0t+a_0},
		$
		which is a contradiction. This establishes that $\displaystyle\sup_{t>0}||u(t)||_{L_{per}^2}=+\infty,$ as requested.
\end{proof}

  \begin{remark} An important point deserves to be mentioned concerning Proposition \ref{bw}. Specifically, if the solution obtained in Proposition $\ref{lwpthm}$ satisfies the blow-up alternative (that is, if $t_{\max}<+\infty$), then the solution $(u,u_t)$ blows up in finite time in the respective norm. Thus,  $\displaystyle\sup_{t \uparrow t_{\max}} ||u(t)||_{L_{per}^2} = +\infty$. If, however, $t_{\max}=+\infty$, then the blow-up occurs at infinite time. This behavior is commonly called a putative blow-up.
    \end{remark}

 \section{Existence of periodic waves}\label{section4} 
 
 \subsection{The mountain pass theorem: existence of weak solutions }

 Although we are considering the periodic traveling wave solution for equation $(\ref{KF3})$, denoted by $\varphi$, we will change this notation to $\rho$ to avoid confusing the reader. For the same reason, we also change the notations for the wave speed $c$ and the parameter $\omega$. In fact, substituting the traveling wave solution of the form $u(x,t)=\rho_\mathtt{c}(x-\mathtt{c}t)$ into $\eqref{KF2}$ for $\alpha=-1$, one has \begin{equation}\label{travSG3}
 	\mathtt{w} \rho_\mathtt{w}''+\sinh(\rho_\mathtt{w})=0,
 \end{equation}where $ \mathtt{w} =\mathtt{w}(\mathtt{c})=1-\mathtt{c}^2 \in (0,1)$ and $0<|\mathtt{c}|<1$. We henceforth write $\rho_\mathtt{w}=\rho$ for simplicity. In this setting, our goal is to obtain $\rho \in H^1_{per,m} \backslash  \{0\}$ satisfying \eqref{travSG3} in the weak sense by using the classical mountain pass theorem. 
 
 To set our problem, let us define a functional  $\mathcal{S}_{\mathtt{w}}:H^1_{per,m} \rightarrow \mathbb{R} $ given by   
 \begin{equation}\label{functionalS}
 	\mathcal{S}_\mathtt{w}(v)= \int_{0}^{L}\left[ \frac{(v')^2}{2}-\frac{1}{\mathtt{w}} \cosh(v) +\frac{1}{\mathtt{w}}\right]dx.
 \end{equation}
 Furthermore, for any $h\in H^1_{per,m} $, we have that 
 \begin{align}\label{functionalS.eq1}
 	\lim_{\varepsilon\rightarrow0} \frac{\mathcal{S}_\mathtt{w}(v+\varepsilon h)-\mathcal{S}_\mathtt{w}(v)}{\varepsilon}=&\lim_{\varepsilon\rightarrow0}\frac{1}{\varepsilon} \int_{0}^{L}\left\{  \varepsilon v'h'+\frac{\varepsilon^2 (h')^2 }{2}  -\frac{1}{\mathtt{w}} \left[\cosh(\varepsilon h+v)-\cosh(v)\right]\right\}dx\nonumber \\ \\=&\int_{0}^{L}  v' h'dx -	\lim_{\varepsilon\rightarrow0}\frac{1}{\varepsilon\mathtt{w}}\int_{0}^{L}\left[\cosh(\varepsilon h+v)-\cosh(v)\right]dx.\nonumber 
 \end{align}
 Using the Taylor expansion of the hyperbolic cosine function, we obtain
 \begin{align}\label{functionalS.eq2}
 	\cosh(\varepsilon h+v)-\cosh(v)=\sinh(v) \varepsilon h+\mathcal{O}((\varepsilon h)^2).  	\end{align}
 It follows from \eqref{functionalS.eq1} and \eqref{functionalS.eq2} that 
 \begin{align*}
 	\lim_{\varepsilon\rightarrow0} \frac{\mathcal{S}_\mathtt{w}(v+\varepsilon h)-\mathcal{S}_\mathtt{w}(v)}{\varepsilon}=&\int_{0}^{L}\left(v'h' -\frac{1}{\mathtt{w}}\sinh(v)h\right)  dx.
 \end{align*}Thus, the Fréchet derivative of $\mathcal{S}_\mathtt{w}$ at $v$ is given by
 \begin{equation}\label{functionalS.eq3}
 	\mathcal{S}_\mathtt{w}'(v) h=(\mathcal{S}_\mathtt{w}'(v), h)_{H^{-1}_{per,m},H_{per,m}^1}=\int_{0}^{L}\left(v'h' -\frac{1}{\mathtt{w}}\sinh(v)h\right)  dx,
 \end{equation}
 where $h\in H^1_{per,m}$. Then, for a fixed $L\in  (0,2\pi)$, the nontrivial critical points of $\mathcal{S}_\mathtt{w}$ are, by $(\ref{functionalS.eq3})$, the weak solutions of \eqref{travSG3}. 
  
 
 \begin{remark}\label{boundedba}
 	
 	It is important to mention that the functional $\mathcal{S}_\mathtt{w}$ in \eqref{functionalS} is unbounded both from below and from above on $H^1_{per,m}$. In fact, for $v(x)=\kappa$, $\kappa\in \mathbb{R}$, we have 
 	\begin{equation*}
 		\mathcal{S}_\mathtt{w}(\kappa) =-\frac{L}{\mathtt{w}} \cosh(\kappa) +\frac{L}{\mathtt{w}}\to -\infty,
 	\  \text{as } \ |\kappa|\to+\infty.
 	\end{equation*}
 	On the other hand, for the sequence $v_n(x)=\sin\left(\frac{2n \pi  x}{L}\right)$, $n\in \mathbb{N}$, we observe that each $v_n$	belongs to $H^1_{per,m}$, $(v_n')^2=\frac{4\pi^2n^2}{L^2} \cos^2\left(\frac{2n\pi x}{L}\right)$ and
 	\begin{equation*}
 		\int_0^L (v_n')^2 dx = \frac{2\pi^2n^2}{L^2}\int_{0}^{L}\left[\cos\left(\frac{4n\pi x}{L}\right)+1\right] dx=\frac{2\pi^2n^2}{L}.
 	\end{equation*}Consequently, since $\cosh(v_n)=\cosh(|v_n|)\leq \cosh(1)$ it follows that
 	\begin{equation*}
 		\mathcal{S}_\mathtt{w}(v_n)= \frac{ \pi^2n^2}{L}-\frac{1}{\mathtt{w}}\int_{0}^{L}\cosh(v_n)dx+\frac{L}{\mathtt{w}}\ge  \frac{ \pi^2n^2}{L}-\frac{L}{\mathtt{w}}\cosh(1) +\frac{L}{\mathtt{w}}\to +\infty,
 	\ \text{as } n\to+\infty.
 	\end{equation*} 		
 	Therefore, $\mathcal{S}_\mathtt{w}$ is unbounded from above and below. This prevents to
 	find a critical point simply by using standard compact techniques (as found in \cite{strauss}). 
 \end{remark}

 To overcome the difficulty presented in Remark $\ref{boundedba}$, we adopt a different approach to guarantee the existence of critical points for the functional $\mathcal{S}_{\mathtt{W}}$ in \eqref{functionalS}. To do so, we employ the mountain pass theorem due to Ambrosetti and Rabinowitz \cite{ambrosetti} (see also \cite{willem}) and, in part, inspired by \cite{lecoz} and \cite{loreno1}, whose statement is presented below.
 
 \begin{defi}\label{PSc}
 	Let $ \mathcal{X}$ be a Hilbert space with norm $||\cdot||$,  $\mathcal{S}\in C^1( \mathcal{X}, \mathbb{R})$, and $\tilde{c}\in \mathbb{R}$. The functional $\mathcal{S}$ satisfies the Palais–Smale condition, denoted by $(PS)_{\tilde{c}}$, if any sequence $(v_n)\subset \mathcal{X}$ such that  $\mathcal{S}(v_n) \to\tilde{c}$ and $\mathcal{S}'(v_n) \to 0$ admits a convergent subsequence.
 \end{defi}
 
 \begin{theorem}\label{teo.Gordon}
 	Let $ \mathcal{X}$ be a Hilbert space with norm $||\cdot||$,  $\mathcal{S}\in C^1( \mathcal{X}, \mathbb{R})$, $e\in  \mathcal{X}$ and $r_e>0$ be such that
 	\begin{equation}\label{hip.Gordon}
 		\|e\|>r_e \quad \text{and} \quad 	 \inf_{\|v\|=r_e}\mathcal{S}(v)>\mathcal{S}(0)\ge \mathcal{S}(e). 
 	\end{equation}
 	If $\mathcal{S}$ satisfies the $(PS)_{\tilde{c}}$ condition with \begin{align}\label{tilde{c}}
 		\tilde{c}=\inf_{\gamma\in\Gamma}\max_{\tau\in[0,1)} \mathcal{S}(\gamma(\tau))\in (0,+\infty),
 	\end{align}where $\Gamma$ is the set of admissible paths\begin{align*}
 		\Gamma=\{\gamma\in C([0,1], \mathcal{X}): \ \gamma(0)=0, \ \gamma(1)=e\},
 	\end{align*}
 	then $\tilde{c}$ is a critical value of $\mathcal{S}$, that is, there exists $ \rho \in \mathcal{X}$ such that $\mathcal{S}(\rho)=\tilde{c}$ and $\mathcal{S}'(\rho)=0$. 
 \end{theorem}
 \begin{proof}
 	See \cite[Theorem 2.1]{ambrosetti} and \cite[Theorem 2.10]{willem}.
 \end{proof}
 
Based on the arguments presented in \cite{lecoz} and \cite{willem}, we first establish that the functional $\mathcal{S}_\mathtt{w}$ has the mountain pass geometry, that is, the condition in \eqref{hip.Gordon} is satisfied, $\Gamma \neq \emptyset$ and $0<\tilde{c}<+\infty$. After that, we ensure that the Palais–Smale condition as in Definition $\ref{PSc}$ is satisfied. To finish, we can apply the mountain pass theorem to conclude the desired result. 
 
 \begin{lemma}\label{lem.montain}
Let $\mathtt{w}\in(0,1)$ be fixed and  consider  $L\in(0,2\pi \sqrt{\mathtt{w}})$. The functional $\mathcal{S}_\mathtt{w}$ defined in \eqref{functionalS} has the mountain pass geometry. 
 \end{lemma}
 \begin{proof}
 	Indeed, using the Taylor formula with the Lagrange remainder to the hyperbolic cosine function around zero, we obtain a constant $\mathtt{M}(L)>0$ depending on $L$ satisfying 
 	\begin{equation}\label{lem.montain.eq2}
 		\cosh(v)-\cosh(0)=\sinh(0) v+\frac{1}{2!} \cosh(0)v^2+\frac{1}{3!}\sinh(0)v^3+r_3(v),
 	\end{equation}where 
 	\begin{equation}\label{lem.montain.eq2'}
 		|r_3(v)|\leq \frac{\mathtt{M}(L)}{4!}|v|^4. 
 	\end{equation}
 	By the continuous embedding of $H^1_{per,m}\hookrightarrow L^4_{per.m}$, we find a constant $\mathtt{M}_0(L)>0$ such that 
 	\begin{equation}\label{lem.montain.eq3}
 		\|v\|_{L^4_{per}}\leq\mathtt{M}_0(L)  	\|v\|_{H^1_{per}}. 
 	\end{equation}  	 	
 	In addition, it follows from the Poincaré–Wirtinger inequality in $H_{per,m}^1$ that
 	\begin{align}  \label{PWineq}
 		\|v\|^2_{L^2_{per}}\leq 	 \biggl(\frac{L}{2\pi}\biggl)^2\|v\|^2_{H^1_{per}}.
 	\end{align} 
 	Substituting \eqref{lem.montain.eq2}-\eqref{PWineq} into \eqref{functionalS}, we obtain, for every $v\in H^1_{per,m}$,
 	\begin{equation}\label{lem.montain.eq4}\begin{array}{llllllll}
 		\mathcal{S}_\mathtt{w}(v)&=&\displaystyle \frac{1}{2} \|v\|^2_{H^1_{per}} -\frac{1}{\mathtt{w}}\int_{0}^{L} [\cosh(v) -1] dx  \\\\ &= & \displaystyle \frac{1}{2} \|v\|^2_{H^1_{per}} -\frac{1}{\mathtt{w}}\int_{0}^{L} \left|\frac{v^2}{2}+r_3(v)\right| dx \\\\ 
 		&\ge &\displaystyle  \frac{1}{2} \|v\|^2_{H^1_{per}} -\frac{1}{2\mathtt{w}}\|v\|^2_{L^2_{per}}  - \frac{\mathtt{M}(L)}{24\mathtt{w}}\|v\|^4_{L^4_{per}} \\\\&\ge &\displaystyle  \frac{1}{2} \|v\|^2_{H^1_{per}} -\frac{1}{2\mathtt{w}}	 \biggl(\frac{L}{2\pi}\biggl)^2\|v\|^2_{H^1_{per}}  - \frac{\mathtt{M}(L)\mathtt{M}_0^4(L)}{24\mathtt{w}}\|v\|^4_{H^1_{per}} \\\\&\ge & \displaystyle \frac{1}{2}\left[1 - \frac{1}{\mathtt{w}}\biggl(\frac{L}{2\pi}\biggl)^2	 \right]\|v\|^2_{H^1_{per}}  - \frac{\mathtt{M}(L)\mathtt{M}_0^4(L)}{24\mathtt{w}}\|v\|^4_{H^1_{per}}.
 	\end{array}\end{equation}
 	Thus, since $L\in(0,2\pi \sqrt{\mathtt{w}})$, we can assert the existence of $\varepsilon>0$ small enough such that
 	\begin{equation}
 		0<	\varepsilon<\sqrt{\frac{12\mathtt{w}}{\mathtt{M}(L)\mathtt{M}_0^4(L)}\left[1 - \frac{1}{\mathtt{w}}\biggl(\frac{L}{2\pi}\biggl)^2	 \right]}.
 	\end{equation}
 	\indent Therefore, for any $ v\in H^1_{per,m}$ satisfying $\|v\| _{H^1_{per}}=\varepsilon$, we have from \eqref{lem.montain.eq4} that $	\mathcal{S}_\mathtt{w}(v) > 0$. This means that 
 	\begin{equation}\label{lem.montain.eq5}
 		\displaystyle	\inf_{\|v\|_{H^1_{per}}=\varepsilon} \mathcal{S}_\mathtt{w}(v)>0=\mathcal{S}_\mathtt{w}(0).
 	\end{equation}
 	In addition, for any $v\in H^1_{per,m}$ and $\tau>0$, 
 	\begin{equation}\label{lem.montain.eq1}
 		\mathcal{S}_\mathtt{w} (\tau v)=\frac{\tau^2}{2} \|v\|^2_{H^1_{per}}+\frac{L}{\mathtt{w}}-\frac{1}{\mathtt{w}}\int_{0}^{L}\cosh(\tau v)dx.
 	\end{equation}
 	Since the hyperbolic cosine function is positive and exhibits exponential-type growth, it follows from \eqref{lem.montain.eq1} that $	\mathcal{S}_\mathtt{w} (\tau v)<0$ if $\tau$ is large enough. Let $\tau_0>0$ be such that  $	\mathcal{S}_\mathtt{w} (\tau_0 v)<0$ and $\|\tau_0v\|_{H^1_{per}}>\varepsilon$. By defining $e=(\tau_0 v) \in H^1_{per,m}$ and $\displaystyle r_e=\varepsilon$, we prove that	   	 	\begin{equation}\label{lem.montain.eq1'} \inf_{\|v\|=r_e}\mathcal{S}_\mathtt{w}(v)>\mathcal{S}_\mathtt{w}(0)\ge \mathcal{S}_\mathtt{w}(e). 
 	\end{equation}
 	\indent Now, let us consider the mapping
 	$\gamma:[0,1]\to H^1_{per,m}$ defined as $\gamma(\tau)=\tau_0 \tau v$. It follows that $\gamma\in C([0,1], H^1_{per,m})$, $\gamma(0)=0$ and $\gamma(1)=e$, proving that $\Gamma$ is nonempty. 	   
 	This implies that for any $\gamma\in \Gamma$, we have $\|\gamma(1)\|_{H^1_{per}} >r_e>0=\|\gamma(0)\|_{H^1_{per}}$. By continuity of $\gamma$, there exists $\tau _\gamma\in (0,1)$ such that $\|\gamma(\tau_ \gamma)\|_{H^1_{per}} =r_e$. Therefore, from \eqref{lem.montain.eq5} we obtain
 	\begin{equation}
 		\max_{\tau\in [0,1]}\mathcal{S}_\mathtt{w}(\gamma(\tau))\ge\mathcal{S}_\mathtt{w}(\gamma(\tau_\gamma))>0, \quad \gamma\in \Gamma,
 	\end{equation}consequently, the mountain pass level satisfies
 	\begin{align}\label{level}
 		\tilde{c}=\inf_{\gamma\in\Gamma}\max_{\tau\in[0,1]} \mathcal{S}_\mathtt{w}(\gamma(\tau))>0.
 	\end{align}
 	Moreover,
 	\begin{equation}\label{lem.montain.eq6} \inf_{\|v\|=r_e}\mathcal{S}_\mathtt{w}(v)\leq\mathcal{S}_\mathtt{w}(\gamma(\tau_\gamma))\leq \max_{\tau\in[0,1]} \mathcal{S}_\mathtt{w}(\gamma(\tau)), \quad\gamma\in \Gamma,\nonumber
 	\end{equation}and by the definition of the infimum, we obtain
 	\begin{equation}\label{lem.montain.eq7} 0<\inf_{\|v\|=r_e}\mathcal{S}_\mathtt{w}(v) \leq \inf_{\gamma\in\Gamma}\max_{\tau\in[0,1]} \mathcal{S}_\mathtt{w}(\gamma(\tau)) =\tilde{c}.
 	\end{equation}In particular, considering the path $\gamma_0(\tau)=\tau e \in \Gamma$ 		  
 	in \eqref{lem.montain.eq7} and using \eqref{lem.montain.eq1}, it follows that
 	\begin{equation}\label{lem.montain.eq8} 0<\inf_{\|v\|=r_e}\mathcal{S}_\mathtt{w}(v) \leq\tilde{c}\leq     \max_{\tau\in[0,1]}\mathcal{S}_\mathtt{w}(\gamma_0(\tau)) =\max_{\tau\in[0,1]}\mathcal{S}_\mathtt{w}(\tau e)\ \leq \frac{1}{2}\|e\|^2_{H^1_{per}}+\frac{L}{\mathtt{w}}<+\infty,
 	\end{equation}
 	 that is, $\mathcal{S}_\mathtt{w}$ has a mountain pass geometry. \end{proof}
 
 \begin{prop}\label{prop.Gordon} Let $\mathtt{w}\in(0,1)$ be fixed and  consider  $L\in(0,2\pi \sqrt{\mathtt{w}})$. Then, the mountain pass level $\tilde{c}$ in \eqref{level} is a critical value of the functional $\mathcal{S}_\mathtt{w}$ defined in \eqref{functionalS}, that is, there exists $\rho\in H^1_{per,m}$ such that $\mathcal{S}_\mathtt{w}(\rho)=\tilde{c}$ and $\mathcal{S}_\mathtt{w}'(\rho)=0$ in $H^{-1}_{per}$. 	 
 \end{prop}
 \begin{proof}
 	In fact, by Lemma $\ref{lem.montain}$, we see that $\mathcal{S}_\mathtt{w}$ has a mountain pass geometry. Using \cite[Theorem 2.9]{willem}, we immediately obtain the existence of a Palais-Smale sequence at the mountain pass level. Therefore, we can apply Theorem \ref{teo.Gordon} in order to conclude the existence of a nontrivial critical value of $\mathcal{S}_\mathtt{w}$, as desired.
 \end{proof}
 
 \begin{remark} It is important to mention that $\rho$ is smooth and solves equation $(\ref{travSG3})$ in the classical sense. Indeed, since $\rho \in H_{per,m}^1$, we obtain that $\sinh(\rho)\in H_{per}^1$, ensuring that equation $(\ref{travSG3})$ holds in the classical sense rather than merely in the weak one.
\end{remark}
\begin{remark}The mountain pass theorem used to prove the existence of solutions 
$\rho$ to equation 
$(\ref{travSG3})$ is, as is well known, a nonstandard technique for obtaining critical points that are indeed solutions of elliptic equations. This method appears to be particularly useful for guaranteeing the existence of such critical points, in contrast with other variational techniques, where one typically seeks suitable minimizers of (a part of) the energy subject to constraints. As far as we can see, the strong nonlinearity $\sinh(\rho)$ in $(\ref{travSG3})$ prevents the use of standard variational techniques to obtain constrained minimizers that solve nonlinear equations. The method developed in this work is potentially applicable to other related equations.
\end{remark}
 
 \subsection{Existence of explicit smooth solutions} \label{sec.exp.sol}
Our intention is to proceed in a different way, by obtaining periodic solutions of equation~\eqref{travSG3}. In the previous subsection, we fixed $\mathtt{w} \in (0,1)$ and considered $L$ depending on $\mathtt{w}$ over the interval $(0,2\pi \sqrt{\mathtt{w}})$, so that $\rho$ is a weak solution of~\eqref{travSG3}, characterized as a critical point of the functional $\mathcal{S}_{\mathtt{w}}$ in~\eqref{functionalS}. In this subsection, our approach is different: for a fixed period $L \in (0,2\pi)$, we consider periodic waves $\varphi$ that depend smoothly on $\omega \in (0,1)$ and solve \eqref{travSG3}. 

In fact, substituting the traveling wave solution of the form $u(x,t)=\varphi_c(x-ct)$ into $(\ref{KF2})$ for $\alpha=-1$, one has
 \begin{equation}
 	c^2\varphi_c''-\varphi_c''-\sinh(\varphi_c)=0.
 	\label{travSG}\end{equation} By considering $\omega=\omega(c)=1-c^2\in (0,1)$, we obtain from $(\ref{travSG})$ that  
 	\begin{equation}\label{travSG1}
 	\omega \varphi_\omega''+\sinh(\varphi_\omega)=0.
 \end{equation}
 \begin{remark}\label{comega} It is important to highlighted that the case $\omega=1$, that is, $c=0$, can be considered in our analysis in the limit sense as $\displaystyle\lim_{c\rightarrow 0}\omega(c)=1$. Since our construction of smooth periodic waves depending on $\omega$ are based on the implicit function theorem, we need to restrict ourselves the analysis in open sets.
 \end{remark}
  We henceforth write $\varphi_\omega=\varphi$ for simplicity. First, multiplying the equation \eqref{travSG1} by $\varphi'$ and integrating the result, we deduce the differential equation to the following quadrature form:
 \begin{equation}\label{travSG2}
 	( \varphi')^2= \frac{2}{\omega}[a-\cosh(\varphi)], 
 \end{equation}
 where $a\in \mathbb{R}$ is a constant of integration. 
 
 Now, let us consider the ansatz given in terms of the hyperbolic tangent function and the Jacobi elliptic function as
 \begin{equation}\label{ansatz}
 	\varphi(x)  =2\operatorname{arctanh}(k \sn (bx ;k ) ),
 \end{equation}
 where $k\in (0,1)$ is the modulus of a snoidal-type elliptic function and $b\in\mathbb{R}$ is a constant to be determined. 
 
 Employing the derivatives $\displaystyle\frac{d}{du}\operatorname{arctanh}(u)=\frac{1}{1-u^2}$ and $\displaystyle\frac{d}{du}\sn(u;k)=\cn(u;k)\dn(u;k)$, in combination with the identities $ \cn^2=(1-\sn^2)$ and $\dn^2=(1-k^2\sn^2)$, upon differentiating \eqref{ansatz}, we find that
 \begin{equation}\label{varphilinha}
 	\varphi'(x) =\frac{2bk}{1-k^2 \sn^2 (bx ;k )} \cn (bx ;k ) \dn (bx ;k )= 2kb\biggl[ \frac{ \cn (bx ;k )}{\dn (bx ;k )}\biggl], 
 \end{equation}which implies that 
 \begin{align}
 	\label{ansatz2}
 	[	\varphi'(x) ]^2= 4k^2b^2 \biggl[\frac{1- \sn^2 (bx ;k )}{1-k^2\sn^2 (bx ;k )}\biggl].  
 \end{align}
\indent  Next, using the expression $\displaystyle\cosh(2\operatorname{arctanh}(u))=\frac{1+u^2}{1-u^2}$, along with equations \eqref{travSG2} and \eqref{ansatz2}, we obtain
 \begin{align*}
 	4k^2b^2 \biggl[\frac{1- \sn^2 (bx ;k )}{1-k^2\sn^2 (bx ;k )}\biggl]=  \frac{2}{\omega}\biggl[a-\frac{1+k^2\sn^2(bx;k)}{1-k^2\sn^2(bx;k)}\biggl],
 \end{align*}that is, 
 \begin{align}
 	\label{ansatz4}
 	2k^2b^2 \omega[ 1- \sn^2 (bx ;k )] =   [ 1+k^2\sn^2(bx;k)]+a[1-k^2\sn^2(bx;k)]. \end{align}
 From \eqref{ansatz4} it follows that
 \begin{align}
 	\label{ansatz5}
 	[	2k^2b^2 \omega+1-a]+[a+1-	2 b^2 \omega ]k^2\sn^2 (bx ;k ) =    0. \end{align}
 According to \eqref{ansatz5}, it is reasonable to assume that $a$ and $b$ satisfy
 \begin{equation*}
 	a=2k^2b^2\omega+1 \quad \text{and} \quad b^2=\frac{1}{\omega(1-k^2)}. 
 \end{equation*}
 Considering the positive root for $b$, we obtain an explicit solution $\varphi$ for \eqref{travSG2}, and consequently to \eqref{travSG1}, given by  
 \begin{equation}	 \label{travSG13}
 	\varphi(x)  =2\operatorname{arctanh}\biggl(k \sn\biggl(\frac{x}{\sqrt{\omega(1-k^2)}} ;k\biggl)\biggl).
 \end{equation}  
 
 Furthermore, since the Jacobi elliptic function of snoidal kind has real period equal to $4K(k)$, we must have
 \begin{equation}\label{speed1}
 	b=\frac{1}{\sqrt{\omega(1-k^2)}}=	\frac{1}{\sqrt{(1-c^2)(1-k^2)}}=\frac{4K(k)}{L}.
 \end{equation}
 
 \indent At this point, it is necessary to ensure that, for fixed $L\in (0,2\pi)$, the condition $0<|c|<1$ holds, or equivalently, $0<\omega< 1 $. In fact, as the complete elliptic integral of the first kind, given by $K(k)=\displaystyle\int_{0}^{\frac{\pi}{2}}\frac{1}{\sqrt{1-k^2\sin^2(\theta)}}d\theta $, it is an increasing function satisfying $K(k)\rightarrow \frac{\pi}{2}$ when $k\rightarrow 0^+$ and $K(k)\rightarrow +\infty$ when $k\rightarrow 1^-$, 
 it follows from the relation in \eqref{speed1}  that
 \begin{equation}\label{speed2}
 	0<\omega=1-c^2=\frac{L^2}{16K^2(k)(1-k^2)} \leq \lim_{k\rightarrow 0^+} \frac{L^2}{16K^2(k)(1-k^2)} =\frac{L^2}{4\pi^2}< 1, 
 \end{equation} 
 as desired. In addition, from $(\ref{speed1})$ we deduce that
 \begin{equation}\label{dk/dw}
 	 \frac{dk}{d\omega}=\left(\frac{d\omega}{dk}\right)^{-1}=\displaystyle\frac{8k(1-k^2)^2K(k)^3}{L^2(K(k)-E(k))}>0,
 \end{equation}
and by the inverse function theorem, we obtain that the mapping $\omega \in (0,1) \mapsto \varphi \in H_{per,m}^{\infty}$ is smooth for a fixed $L \in (0,2\pi)$.
 
 The following result can be enunciated.
 
 \begin{proposition}\label{prop.WS}
 	Let $L\in (0,2\pi)$ be fixed. There exists a smooth curve of periodic traveling wave solutions for the equation $(\ref{travSG1})$, given by
 	\begin{equation}
 		\displaystyle
 		\omega\in\left(0,1\right)
 		\mapsto \varphi\in H_{per,m}^\infty([0,L]),\label{curveeqSG1}\end{equation}
 	where
 	\begin{equation} 
 		\varphi(x)=2\operatorname{arctanh}\biggl(k \sn\biggl(\frac{4K(k)}{L} \ x ;k\biggl)\biggl).\label{existSG1}
 	\end{equation}  
 \end{proposition}
 \begin{flushright}
 	${\blacksquare}$
 \end{flushright}

 \begin{remark}\label{smoothcomega}
Let $L\in (0,2\pi)$ be fixed. Proposition $\ref{prop.WS}$ states that $\varphi$ depends smoothly on the parameter $\omega\in(0,1)$. A natural question that arises is the smooth dependence of $\varphi$ with respect to $c\in(-1,1)$, since $\omega=1-c^2$. In fact, for $c\in (-1,0)\cup(0,1)$ it is easy to see that $\varphi$ depends smoothly on $c$ because of the chain rule $\displaystyle\frac{\partial\varphi}{\partial c}=-2c\frac{\partial \varphi}{\partial\omega}$. At $c=0$, we obtain that $\displaystyle\lim_{c\rightarrow 0}\frac{\partial\varphi}{\partial c}=\lim_{c\rightarrow 0}\left(-2c\frac{\partial \varphi}{\partial\omega}\right)=0$, so that $\varphi$ is smooth at $c=0$. Therefore, $\varphi$ can be considered as a smooth function in terms of $c\in(-1,1)$ for all $L\in (0,2\pi)$ fixed.
 
 \end{remark}
 
 \indent We now establish an important result showing that the periodic solution $\varphi$ in \eqref{existSG1} may exhibit instabilities.
 
  \begin{proposition}\label{propEphi} Let $\varphi$ be the periodic solution in $(\ref{existSG1})$. We have that $\mathcal{E}(\varphi,c\varphi')$ has indefinite sign. 
 \end{proposition}
 \begin{proof}	In fact, from \eqref{travSG2} we have the quadrature form
 	\begin{equation}\label{quadra1}
 		( \varphi')^2= \frac{2}{\omega}[a-\cosh(\varphi)], 
 	\end{equation}
 	where $a=\displaystyle\frac{k^2+1}{1-k^2}$ and $k\in (0,1)$. Thus, by \eqref{quadra1}, it follows that 	
 	 	\begin{equation}\begin{array}{lllll} 
 			\mathcal{E}(\varphi, c\varphi')
 			&=&\displaystyle\frac{1}{2}\int_{0}^{L} \left[(\varphi ')^2+(c\varphi')^2-2(\cosh(\varphi)-1)\right] dx\\\\
 			&= & \displaystyle\frac{1}{2}\int_{0}^{L} \left[(1+c^2)\frac{2}{\omega}[a-\cosh(\varphi)]-2(\cosh(\varphi)-1)\right] dx\\\\
 			&=& \displaystyle \int_{0}^{L} \left\{\left[1+(1+c^2)\frac{1}{\omega}a\right]-\left[ 1+(1+c^2)\frac{1}{\omega}\right]\cosh(\varphi) \right\}dx\\\\
 			& = & \displaystyle \int_{0}^{L} \left\{\left[1+ \frac{(1+c^2)(1+k^2)}{(1-c^2)(1-k^2)}\right]-\left[1+\frac{1+c^2}{1-c^2}\right]\cosh(\varphi) \right\} dx.
 	\end{array}\end{equation}
 	Using the solution in \eqref{ansatz} and the identity $\displaystyle\cosh(2\operatorname{arctanh}(u))=\frac{1+u^2}{1-u^2}$, we obtain that 
 	\begin{align} \label{blowup5}
 		\mathcal{E}(\varphi, c\varphi')
 		=&  \left( 1+ \frac{(1+c^2)(1+k^2)}{(1-c^2)(1-k^2)}\right) L -\left(1+\frac{1+c^2}{1-c^2}\right)\int_{0}^{L} \frac{1+k ^2\sn (bx ;k )^2}{1-k^2 \sn (bx ;k )^2}  dx ,
 	\end{align}
 	where $b^2=\displaystyle\frac{1}{\omega(1-k^2)}$.   We can rewrite the integral in \eqref{blowup5} as
 	\begin{align} \label{blowup6}
 		\int_{0}^{L} \frac{1+k ^2\sn (bx ;k )^2}{1-k^2 \sn (bx ;k )^2}  dx =	\int_{0}^{L} \frac{2-\dn (bx ;k )^2}{ \dn (bx ;k )^2}  dx =-L+\frac{2}{b}	\int_{0}^{bL} \frac{1}{ \dn (u ;k )^2}  du .
 	\end{align} From \cite[Formula 315.02]{byrd} it follows that	
 	\begin{equation*}
 		\int \frac{1}{\operatorname{dn}^2( u)} \, du 
 		=  \frac{1}{1-k^2} \left[ E(u) - k^2 \,  \, \frac{\sn( u)\cn(u) u }{\dn(u)}  \right],
 	\end{equation*}
 	where $$E(u)=E(\psi, k)=\displaystyle\int_{0}^{u } \dn^2 (x) \, dx=\int_{0}^{\psi} \sqrt{1 - k^2 \sin^2 (\vartheta)} \, d\vartheta , \quad \sn (u)=\sin(\psi),$$ denotes the normal elliptic integral of the second kind (see \cite[Formula 110.03]{byrd}). Thus, 
 	\begin{align}\nonumber
 		\int_{0}^{bL} \frac{1}{ \dn (u ;k )^2}  du = \frac{1}{1-k^2} \left[ E(bL) - k^2 \,  \, \frac{\sn( bL)\cn(bL)bL}{\dn(bL)}  \right]. 
 	\end{align}Since $bL=4K(k)$ and $\sn(u+4K(k))=\sn(u)$ we have that $\sn( bL)=\sn(0)=0 $. Therefore, 	\begin{align}\label{blowup7'}
 		\int_{0}^{bL} \frac{1}{ \dn (u ;k )^2}  du = \frac{1}{1-k^2}   E(4K(k))  . 
 	\end{align}
 	From \eqref{blowup5}, \eqref{blowup6} and \eqref{blowup7'} we obtain 
 	\begin{equation}\label{energyphi}\begin{array}{lllll}
 		\mathcal{E}(\varphi, c\varphi')
 		&=& \displaystyle \left[ 1+ \frac{(1+c^2)(k^2+1)}{(1-c^2)(1-k^2)} \right] L -\left(1+\frac{1+c^2}{1-c^2}\right) \left[ -L+\frac{2}{b}	\int_{0}^{bL} \frac{1}{ \dn (u ;k )^2}  du \right ]\\\\ 
 		&= &\displaystyle \left[  \frac{2-k^2(1-c^2)}{(1-c^2)(1-k^2)} \right] 2L-4b E(4K(k)) \\\\ 
 		&= &\displaystyle \left[  \frac{2-k^2(1-c^2)}{(1-c^2)(1-k^2)} \right] 2L-\frac{16K(k)}{L} E(4K(k)) \\\\
 		&= &\displaystyle \left[  \frac{2 }{(1-c^2)(1-k^2)}-\frac{ k^2 }{ (1-k^2)} \right] 2L-\frac{16K(k)}{L} 4E(k)  \\\\ 
 		&= &\displaystyle   \frac{64K^2(k) }{L}-\frac{2L k^2 }{ (1-k^2)}   -\frac{64K(k) E(k) }{L} \\\\
 		&=& \displaystyle\frac{64}{L}\left[ K^2(k)-K(k) E(k)  -\frac{L^2 k^2 }{ 32(1-k^2)}   \right].
 \end{array}	\end{equation}
 	\indent In order to describe the behavior of $\mathcal{E}(\varphi,c\varphi')$ give by $(\ref{energyphi})$, we need to study the sign of the quantity $ \displaystyle K^2(k)-K(k) E(k)  -\frac{L^2 k^2 }{ 32(1-k^2)}$. In fact, if $L\in (0,l]$, where $l>0$ is small, we obtain that $\mathcal{E}(\varphi,c\varphi')$ is positive. Now, if $L$ is taken in the whole interval $(0,2\pi)$, $\mathcal{E}(\varphi,c\varphi')$ assumes positive and negative values (see Figure \ref{fig1}).

 	\begin{figure}[h!] 
 		\centering
 		\includegraphics[scale=0.4,angle=0]{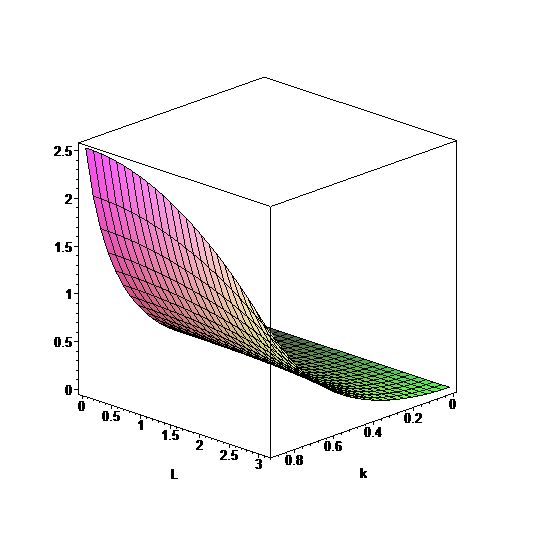}
 		\includegraphics[scale=0.4,angle=0
 		]{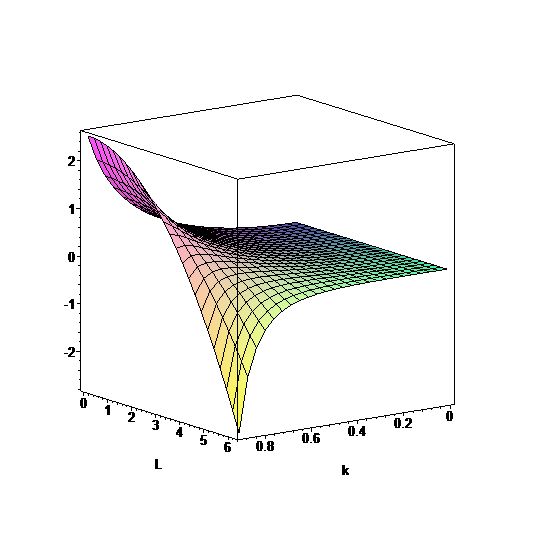} 
 		\caption{Left: Graph of $\mathcal{E}(\varphi, c\varphi')$ for $L\in (0,\pi]$ and $k\in(0,1)$. Right: Graph of $\mathcal{E}(\varphi,c\varphi')$ for $L\in (0,2\pi)$ and $k\in (0,1)$.}\label{fig1}
 	\end{figure}
 \end{proof}
 
\begin{remark}\label{remEphiphi} Some important remarks concerning Proposition \eqref{propEphi} should be highlighted. Indeed, for a fixed $L \in (0,2\pi)$, the sign of the energy of the periodic wave $(\varphi, c\varphi')$ depends on the modulus $k \in (0,1)$. According to Figure \ref{fig1}, we observe that, as $L$ approaches $2\pi$, the function $\mathcal{E}(\varphi, c\varphi')$ becomes negative for all values of the modulus $k \in (0,1)$. If $L$ is small, this quantity is positive. As far as we can see, the indefinite sign of $\mathcal{E}(\varphi, c\varphi')$ implies, according to Proposition \ref{bw}, the occurrence of a blow-up phenomenon for perturbations of the traveling wave generated by $\varphi$. This, in turn, suggests an instability result for the pair $(\varphi, c\varphi')$ for some values of $k\in (0,1)$.
\end{remark} 
 	
\section{Spectral analysis}\label{spectralKG}  	
 It is well known that equation \eqref{travSG1} is conservative, and its solutions lie on the level curves of the energy
 	\begin{equation}\label{hamilt.energ}
 		\mathcal{H}(\varphi, \varphi')=\frac{(\varphi')^2}{2} +\frac{1}{\omega} \cosh(\varphi)-\frac{1}{\omega}. \end{equation}
 	We see that $\varphi$ is a periodic solution of equation \eqref{travSG1} if, and only if, $(\varphi, \varphi')$ is a periodic orbit of the planar differential system 
 	\begin{align}\label{sist.hamil}
 		\begin{cases}
 			\varphi'=	\chi, \\ 
 			\chi'=- \frac{1}{\omega} \sinh(\varphi),
 		\end{cases} 	
 	\end{align}where $(0,0)$ is the unique critical point. It is worth noting that the condition $\omega>0$ was imposed to ensure that the critical point is a center.
 	
 	The periodic orbits corresponding to system \eqref{sist.hamil} are contained within appropriate level sets of the energy $	\mathcal{H}$ in \eqref{hamilt.energ}. In other words, the pair $(\varphi, \varphi')$
 	satisfies the equation $	\mathcal{H}(\varphi, \varphi')=B$ for all $B\in (0,+\infty)$. In addition, the periodic orbits of the planar system \eqref{sist.hamil} corresponds to an odd periodic solution $\varphi$ of the equation \eqref{travSG1}. Furthermore, the periodic solutions have period $\ell$, which depends on $B$ and $\omega$. However, since our intention is first to study the behavior of $\ell$ only in terms of $B$, we 
 	(formally) define it as follows:
 	\begin{equation}\label{period}
 		\ell(B)=\frac{\sqrt{2} }{2} \displaystyle \int_{\varphi ^-}^{\varphi^+}\frac{d\varphi}{\sqrt{B+\frac{1-\cosh(\varphi )}{\omega}}},
 	\end{equation} 
 	where $\varphi^{-}$ and $\varphi^{+}$ are the roots of $ B+\frac{1-\cosh(\varphi )}{\omega}$ and they are  given by  $$\varphi ^-=\displaystyle\min_{x\in[0,\ell]}\varphi (x) \quad \text{and} \quad\varphi ^+=\displaystyle\max_{x\in[0,\ell]}\varphi (x).$$  
 	
 	\indent On the other hand, consider $\Gamma_B$ as the orbit in the phase portrait corresponding to the periodic solution $\varphi$ in \eqref{existSG1}.  
 	The mapping in $(\ref{period})$ can be seen as a smooth function in terms of $B$ and $\omega$ and it can rewritten as
 	\begin{align}\label{period.function}
 	 \ell(B):(0,+\infty)&\rightarrow (0,2\pi) \nonumber\\ B \ \ &\mapsto \displaystyle\int_{\Gamma_B} \frac{d\varphi }{\varphi '}.
 	\end{align} We call $\ell=\ell(B)$ as the \textit{period function} associated with the periodic solution $\varphi$. Concerning the monotonicity of the period function $\ell$ in terms of $B\in (0,+\infty)$, we have the following result:
 	
 		\begin{lema}
 		\label{lem.period.function} The period function in \eqref{period.function} satisfies $ 	\ell'(B)<0$ for all $B\in(0,+\infty)$. 
 	\end{lema}
 	\begin{proof}
 		Let us consider the functions $	\displaystyle\mathtt{g}(x)=\frac{1}{\omega} \sinh(x)$ and  $\displaystyle	G(x)=\frac{1}{\omega} \cosh(x) -\frac{1}{\omega}$, defined on $\mathbb{R}$. We can rewrite $\mathtt{g}$  as $\mathtt{g}(x)= x \mathtt{h}(x)$, where $\mathtt{h}$ is a positive 
 		smooth function given by 
 		\begin{align*}
 		\mathtt{h}(x)=\begin{cases}
 				\frac{1}{\omega}, \quad 	&x=0, \\  \frac{1}{\omega}\frac{\sinh(x)}{x}, \quad &x\neq 0.
 			\end{cases}
 		\end{align*}Furthermore, we observe that
 		$\displaystyle\lim_{x\rightarrow\pm +\infty}G(x)=+\infty,$
 		and $x \mathtt{g}''(x)= \frac{1}{\omega}\sinh(x) x >0 $ for all $x\neq 0$. Therefore, according to \cite[Corollary 2.3]{chowwang}, the period function $ \ell(B)$ is monotone decreasing on $(0,+\infty)$.
 	 	\end{proof}
 	
The main objective of this section is to study the non-positive spectrum corresponding to the operator
  $\mathcal{L}:H^2_{per}([0,\ell])\times H^1_{per}([0,\ell]) \subset\mathbb{L}^2_{per}([0,\ell])\rightarrow \mathbb{L}^2_{per}([0,\ell])$ defined in \eqref{matrixop313}. The eigenvalue problem related to the linearized operator 
 	\begin{equation}\label{opL1}
 		\mathcal{L}_1=-\omega\partial_x^2-\cosh(\varphi),
 	\end{equation}
 	associated with the solution in \eqref{existSG1}, 
 	will be studied through the analysis of the monotonicity of the period map corresponding to Newton’s equation in \eqref{travSG1}. This study aims to determine the spectral properties related to the operator $\mathcal{L}_1$.

 	Before proceeding, we briefly summarize some foundational elements of Floquet theory (for details, see \cite{eas}, \cite{magnus} and \cite{neves}). Let $\mathcal{Q}$ be an even smooth function that is $\ell$-periodic. Define the Hill operator $\mathcal{P}$ over the space $L^2_{per}([0,\ell])$, with domain $D(\mathcal{P})=H^2_{per}([0,\ell])$, as follows:
 	\begin{equation*}
 		\mathcal{P}=-\partial_x^2+\mathcal{Q}(x).
 	\end{equation*}
 	The spectrum of $ \mathcal{P}$ consists of a divergent sequence of real eigenvalues
 	\begin{equation}\label{specP}\lambda_0<\lambda_1\leq\lambda_2 \leq\lambda_3 \leq\lambda_4 \leq \cdots \leq \lambda_{2n-1}\leq \lambda_{2n}\leq \cdots ,\end{equation}
 	where equality $\lambda_{2n-1}= \lambda_{2n}$ corresponds to a double eigenvalue. Moreover, the spectrum can be characterized by the number of zeros of the corresponding eigenfunctions: if 
 	$f$ is an eigenfunction associated with either $\lambda_{2n-1}$ or $\lambda_{2n}$, then $f$ possesses exactly $2n$ zeros within the half-open interval 
 	$[0,\ell)$. In particular, the eigenfunction corresponding to the simple eigenvalue $\lambda_0$ is strictly positive (or strictly negative), exhibiting no zeros in $[0,\ell)$.
 	
 	Next, consider $z(x)$ a nontrivial $\ell$-periodic solution of the equation 
 	\begin{equation}\label{eq.floquet}
 		\mathcal{P}f=-f''+\mathcal{Q}(x)f=0.
 	\end{equation}
 	Let $y(x)$ be another solution of the equation \eqref{eq.floquet} which is linearly independent with $z(x)$. According to Floquet theory \cite{magnus}, there exists a real constant $\theta $ such that 
 	\begin{equation}\label{theta.param}
 		y(x+\ell)=y(x)+\theta z(x).
 	\end{equation} 
 	We see that	$\theta =0$, if and only if, $y$ is $\ell-$periodic.\\
 	\indent We now return our focus to the eigenvalue problem associated with the Hill operator in \eqref{opL1}. First, by differentiating equation \eqref{travSG1}, we observe that $\varphi' \in \Ker(\mathcal{L}_1)$. Furthermore, $\varphi'$ has exactly two zeros in the half-open interval $[0,\ell)$. Based on our construction, for the modified operator
 	\begin{equation}\label{operadorL2}
 		\mathcal{P} = \frac{1}{\omega}\mathcal{L}_1 = -\partial_x^2 - \frac{1}{\omega}\cosh(\varphi),
 	\end{equation}
 	we have that zero is either the second or the third eigenvalue of $\mathcal{P}$. Since $\omega$ is positive, $\mathcal{L}_1$ preserves the same position of the zero eigenvalue as $\mathcal{P}$.

 		\begin{lema}\label{theta}
 		We have that $ \theta=-\ell'(B)$, where $\theta$ is the constant in \eqref{theta.param}.
 	\end{lema}
 	\begin{proof}
 	Consider $\varphi$ as the solution of \eqref{travSG1}. Let $y$ be the unique solution of the corresponding initial value problem
 	\begin{equation}\label{Cauchyproblem0} 
 		\begin{cases}
 			-	 y''-\frac{1}{\omega}\cosh(\varphi) y=0, \\
 			y(0)= 0, \\
 			y'(0)=\frac{1}{\varphi'(0)}.
 		\end{cases}
 	\end{equation}
 From $(\ref{travSG1})$, we see that $\varphi'$ is an $\ell$-periodic solution of the equation in \eqref{Cauchyproblem0}. As in $(\ref{theta.param})$, there exists a constant $\theta$ such that
 	\begin{equation}\label{eq2.Cauchyproblem0}
 		y(x+\ell(B)) = y(x) + \theta \varphi'(x). 
 	\end{equation} 	
 	By taking $x = 0$ in $(\ref{eq2.Cauchyproblem0})$, it follows that
 	\begin{equation}\label{eq3.Cauchyproblem0}
 		\theta = \frac{y(\ell(B)) }{\varphi'(0) }. 
 	\end{equation}
 	
 	Now, since $\varphi$ is odd and periodic, we have $\varphi(0) = \varphi(\ell(B)) = 0$. Thus, the smooth dependence of the solution $\varphi$ on the initial data implies that $\varphi$ depends smoothly on the parameters $B$ and $\omega$. In particular, this allows us to differentiate $\varphi(\ell(B)) = 0$ with respect to $B$ to obtain
 	\begin{equation}\label{eq5.Cauchyproblem0}
 		\varphi'(\ell(B)) \ell'(B) + \frac{\partial \varphi(\ell(B))}{\partial B} = 0.  
 	\end{equation}
 	
 	Next, we return to the quadrature form in \eqref{hamilt.energ}, which satisfies
 	\begin{equation}\label{eq6.Cauchyproblem0}
 		\frac{(\varphi'(x))^2}{2} + G(\varphi(x)) = B,
 	\end{equation}
 	where $\displaystyle G(\varphi(x)) = \frac{1}{\omega}\cosh(\varphi(x)) - \frac{1}{\omega}$. Differentiating equation \eqref{eq6.Cauchyproblem0} with respect to $B$ and evaluating the final result at $x = 0$, we obtain
 	\begin{equation*}
 		\varphi'(0) \frac{\partial \varphi'(0)}{\partial B} + \frac{1}{\omega}\sinh(\varphi(0)) \frac{\partial \varphi(0)}{\partial B} =\varphi'(0) \frac{\partial \varphi'(0)}{\partial B}= 1,
 	\end{equation*}
 	which implies that $\frac{\partial \varphi'(0)}{\partial B} = \frac{1}{\varphi'(0)}$. In addition, since in particular $\frac{\partial \varphi }{\partial B}$ is odd, it follows that $\frac{\partial \varphi(0)}{\partial B}=0$. On the other hand, differentiating equation \eqref{travSG1} with respect to the parameter $B$, we find that $\frac{\partial \varphi(x)}{\partial B}$ is a solution of the IVP \eqref{Cauchyproblem0}. By the existence and uniqueness theorem, it follows that $\frac{\partial \varphi(x)}{\partial B} = y(x)$.
 	
 These pieces of information, combined with \eqref{eq5.Cauchyproblem0}, yield
 \begin{equation}\label{eq8.Cauchyproblem0}
 	\ell'(B) = -\frac{y(\ell(B))}{\varphi'(\ell(B))} = -\frac{y(\ell(B))}{\varphi'(0)},
 \end{equation}
 in view of the fact that $\varphi'$ is periodic. Combining \eqref{eq3.Cauchyproblem0} and \eqref{eq8.Cauchyproblem0}, we obtain $\theta = -\ell'(B)$. This relation indicates that $y$ is not periodic whenever $\ell'(B) \neq 0$.
 		
 \end{proof}
 	\begin{remark} Some relevant comments are worth emphasizing regarding Lemmas \ref{lem.period.function} and \ref{theta}. Indeed, by Lemma $\ref{lem.period.function}$, we have $\ell'(B) < 0$ for all $B \in (0, +\infty)$, and by Lemma $\ref{theta}$, it follows that $y$ is not periodic for any value of $B$. Moreover, since $\ell$ depends smoothly on $B$ and $\omega$, and $\ell$ is strictly decreasing in terms of $B$, we deduce that for a fixed $L \in (0, 2\pi)$, there exists $B_0 \in (0, +\infty)$ such that $\ell(B_0, \omega) = L$ for all $\omega \in (0, 1)$. This fact allows us to recover the result established in Proposition $\ref{prop.WS}$, which asserts the existence of a smooth curve $\omega \in (0, 1) \mapsto \varphi \in H_{per,m}^{\infty}([0, L])$ with fixed period, obtained here by a different approach.	 		
 	\end{remark}
 	
 	The next result characterizes the spectral behavior of $\mathcal{L}_1$ and links the value of $\theta$ in $(\ref{eq2.Cauchyproblem0})$ to the number of negative eigenvalues. In what follows, we consider the periodic function $\varphi$ with fixed period as determined by Proposition $\ref{prop.WS}$. 
 	
 	\begin{prop}\label{eigenvalues.L1} Let $L\in (0,2\pi)$ be fixed. The linearized operator $ \mathcal{L}_1 : L^2_{\mathrm{per}} \to L^2_{\mathrm{per}}$ defined in \eqref{opL1} admits two simple negative eigenvalues and a simple zero eigenvalue with eigenfunction $\varphi'$.  Moreover, the remainder of the spectrum is a discrete set and bounded away from zero. 
 	\end{prop}
 	\begin{proof}
 		Indeed, Lemmas \ref{lem.period.function} and \ref{theta} yield that $\theta > 0$ on $(0,+\infty)$. Consequently, by \cite[Theorem 3.1]{neves}, zero is the third eigenvalue of $\mathcal{L}_1$, and this eigenvalue is simple. The remaining conclusions follow from the fact that $\frac{1}{\omega}\mathcal{L}_1$ is a Hill operator with an even periodic potential (see \eqref{specP}).
 	\end{proof}
 	
 	Next, we prove that the kernel of the full linear operator $\mathcal{L}$ defined in \eqref{matrixop313} is one-dimensional. 
 	\begin{lema}\label{kernel.L}
 		Let $L\in (0,2\pi)$ be fixed. Let $\varphi$ be the solution obtained in Proposition \ref{prop.WS}. Then  $\Ker(\mathcal{L})=[(\varphi', c\varphi'')]$.
 	\end{lema}
 	\begin{proof}
 		Observe that $(f_1,f_2)\in \Ker(\mathcal{L})$ if, and only if, \begin{align}\label{sistem}
 			\begin{cases}
 				&	 -f_1''-\cosh(\varphi)f_1+cf_2'=0, \\& -cf_1'+f_2=0.
 			\end{cases}
 		\end{align}
 		Substituting the second equation in \eqref{sistem} into the first one, we obtain $\mathcal{L}_1 f_1=0$, that is, $f_1\in \Ker(\mathcal{L}_1)$. It follows from Proposition \ref{eigenvalues.L1} that there exists $\alpha_0 \in \mathbb{R}$ such that $f_1=\alpha_0 \varphi'$. Thus, we have that $$(f_1,f_2)=(\alpha_0 \varphi',\alpha_0c \varphi'')=\alpha_0(\varphi',c\varphi'')\in [(\varphi', c\varphi'')].$$ This proves that $\Ker(\mathcal{L}) \subseteq [(\varphi', c\varphi'')]$. On the other hand, upon differentiating \eqref{travSG1} with respect to $x$ and using the linearity of $\mathcal{L}$, we obtain the opposite inclusion, which completes the proof.
 	\end{proof}

 		The rest of the spectrum of $\mathcal{L}$ is characterized as stated in the next result.
 		
 		\begin{prop}\label{2eigenvalues} Let $L\in (0,2\pi)$ be fixed.
 			The operator $\mathcal{L}$ in \eqref{matrixop313} has exactly two negative eigenvalues, which are simple. Zero is a simple eigenvalue with associated eigenfunction $(\varphi',c\varphi'')$. In addition, the rest of the
 			spectrum is constituted by a discrete set of eigenvalues.
 		\end{prop}  
 		\begin{proof} First, since $\mathcal{L}$  is self-adjoint, we obtain that all its eigenvalues are contained in the real line. Because of the compact embedding $H_{per}^2\times H_{per}^1\hookrightarrow \mathbb{L}_{per}^2$, we obtain that the essential spectrum of $\mathcal{L}$ is empty, so that the spectrum of $\mathcal{L}$ is constituted by a discrete set of eigenvalues with finite multiplicity. Next, for $c = 0$ we see that $\omega =  1$ (see Remark $\ref{comega}$) and
 			\begin{equation}\label{matrixop3130}
 				\displaystyle \mathcal{L}=\left(
 				\begin{array}{ccc}
 					-\partial_x^2-\cosh(\varphi) & &0\\\\
 					\ \ \ \ 0 & & 1
 				\end{array}\right)=\left(
 				\begin{array}{ccc}
 					\mathcal{L}_1 & &0 \\ \\ 
 					0 & & 1
 				\end{array}\right).
 			\end{equation}
 			Since the operator in \eqref{matrixop3130} is diagonal and, by Proposition \ref{eigenvalues.L1}, the linearized operator $\mathcal{L}_1$ admits two simple negative eigenvalues, it follows that $\mathcal{L}$ in \eqref{matrixop3130} also has exactly two simple negative eigenvalues. Then, by Lemma \ref{kernel.L} together with the continuity of the eigenvalues with respect to $c \in (-1,1)$ (see \cite[Chapter 4, Section 3.5]{kato}), it follows that the full linear operator $\mathcal{L}$ in $(\ref{matrixop313})$ admits two simple negative eigenvalues, as desired. 
 		\end{proof}
 		
 		The next step is to study the spectrum of the projection operator $\mathcal{L}_{\Pi}$ in $(\ref{opconstrained2})$. In fact, we start by establishing some essential preliminary facts. First, we consider the constrained subspace $S_1= [1] \subset \Ker(\mathcal{L}_1)^{\perp}=[\varphi']^{\bot}$, which is related to the auxiliary linear operator $\mathcal{L}_{1\Pi}: H^2_{per,m}\subset  L^2_{per,m}\rightarrow  L^2_{per,m}$, defined by $\displaystyle\mathcal{L}_{1\Pi}=\mathcal{L}_1+\frac{1}{L}(\cosh(\varphi),\cdot)_{L_{per}^2}$. Since the kernel of $\mathcal{L}_1$ is simple and generated by $\varphi'$, we can define the number $D_1=(\mathcal{L}_1^{-1} 1,1)_{L^2_{per}},$ because $1\in \Ker(\mathcal{L}_1)^{\bot}$.
 		
 	The Index Theorem, as stated in \cite[Theorem 5.3.2]{kapitula} and \cite[Theorem 4.1]{pel-book}, applies to self-adjoint operators and can be used to determine the precise number of negative eigenvalues and the dimension of the kernel of $\mathcal{L}_{1\Pi}$ and $\mathcal{L}_{\Pi}$. Indeed, as $\Ker(\mathcal{L}_1)=[\varphi']$, one has 
 		\begin{equation}\label{indexformula12}
 			\text{n}(\mathcal{L}_{{1\Pi}})=\text{n}(\mathcal{L}_1)-{\rm n}_0-{\rm z}_0 \quad \text{and} \quad \text{z}(\mathcal{L}_{{1\Pi}})=\text{z}(\mathcal{L}_1)+{\rm z}_0,
 		\end{equation}
 		where $\text{n}(\mathcal{P})$ and $\text{z}(\mathcal{P})$ refer to the count of negative eigenvalues and the kernel dimension of a linear operator $\mathcal{P}$ (counting multiplicities). Furthermore, the numbers ${\rm n}_0$ and ${\rm z}_0$ are defined respectively as
 		\begin{equation}\label{n0z0}
 			{\rm n}_0=
 			\begin{cases}
 				1, \: \text{if} \: D_1<0, \\
 				0, \: \text{if} \: D_1 \geq 0,\ \
 			\end{cases}
 			\quad \text{and} \qquad\ 
 			{\rm z}_0=
 			\begin{cases}
 				1, \: \text{if} \: D_1=0, \\
 				0, \: \text{if} \: D_1 \neq 0.
 			\end{cases}
 		\end{equation}
 		\indent We now present a result that determines the spectral information of the operator $\mathcal{L}_{\Pi}$.
 		
 		\begin{proposition}\label{leman1}
 			Let $L\in (0,2\pi)$ be fixed. The linear operator $\mathcal{L}_{\Pi}$ in $(\ref{opconstrained2})$ has exactly one negative eigenvalue, which is simple, and $(\varphi',c\varphi'')$ is a simple eigenfunction associated with the zero eigenvalue. 
 		\end{proposition}
 		\begin{proof}
 			Counting the negative eigenvalues requires observing that $\mathcal{L}_{\Pi}$ is the constrained operator $\mathcal{L}$ defined in $\mathbb{L}_{per,m}^2$ satisfying $\mathcal{L}_{\Pi}|_{S^\perp}=\mathcal{L}$, with constrained space given by $$S= [(1,0),(0,1)] \subset \Ker(\mathcal{L})^{\perp}=[(\varphi',c\varphi'')]^\perp.$$
 			\indent In association with the constrained set $S$, we define the matrix
 			\begin{equation}\label{matrixD}
 				D=\left[\begin{array}{llll}(\mathcal{L}^{-1}(1,0),(1,0))_{\mathbb{L}^2_{per}}& & (\mathcal{L}^{-1}(1,0),(0,1))_{\mathbb{L}^2_{per}}\\\\
 					(\mathcal{L}^{-1}(1,0),(0,1))_{\mathbb{L}^2_{per}}& & (\mathcal{L}^{-1}(0,1),(0,1))_{\mathbb{L}^2_{per}}\end{array}\right].
 			\end{equation}
 			Given that $\mathcal{L}$ is a linear operator, $\mathcal{L}(0,1)=(0,1)$ and $(1,1)=(1,0)+(0,1)$, we obtain the identities $\mathcal{L}^{-1}(1,0)=\mathcal{L}^{-1}(1,1)-\mathcal{L}^{-1}(0,1)=\mathcal{L}^{-1}(1,1)-(0,1)$, and
 			\begin{equation}\label{D}
 				\left( \mathcal{L}^{-1}(1,0), (1,0) \right)_{\mathbb{L}^2_{per}}= (\tilde{f},1)_{L^2_{per}}  = ( \mathcal{L}_1^{-1} 1,1)_{L^2_{per}}=D_1.
 			\end{equation}
 			Moreover, 
 			\begin{equation}\label{D12}
 				\left( \mathcal{L}^{-1}(1,0), (0,1) \right)_{\mathbb{L}^2_{per}}=\left(\mathcal{L}^{-1}(1,1),(0,1)\right)_{\mathbb{L}^2_{per}}-\left(\mathcal{L}^{-1}(0,1),(0,1)
 				\right)_{\mathbb{L}^2_{per}}=L-L=0,\end{equation}
 			and 
 			\begin{equation}\label{D123}
 				\left( \mathcal{L}^{-1}(0,1), (0,1) \right)_{\mathbb{L}^2_{per}}=\left((0,1),(0,1)\right)_{\mathbb{L}^2_{per}}=L.\end{equation}
 			Hence, from \eqref{D}, \eqref{D12} and \eqref{D123}, $D$ in \eqref{matrixD} takes the form $D=\displaystyle\left[\begin{array}{cc}D_1&  0\\
 				0&  L\end{array}\right]$, where $D_1=(\mathcal{L}_{1}^{-1}1,1)_{L_{per}^2}$. Calculating $D_1$	requires finding an element $\tilde{f}\in H_{per}^2$ such that $\mathcal{L}_1\tilde{f}=1$.	
 			
 			To construct a suitable periodic $\tilde{f}$, we start by analyzing the IVP introduced  in \eqref{Cauchyproblem0} with solution $y$. Furthermore, using the variation of parameters formula, we deduce that
 			$$
 			\mathtt{p}(x)=\frac{1}{\omega}\left(\int_0^xy(s)ds\right)\varphi'(x)-\frac{1}{\omega}y(x)\varphi(x),
 			$$ 
 			is a solution to the equation
 			$$
 			-\omega 	\mathtt{p}''-\cosh(\varphi)	\mathtt{p}=1.
 			$$
 			
 			We will demonstrate that $	\mathtt{p}$ is in fact an $L$-periodic function, which naturally leads to the definition of $\tilde{f}=	\mathtt{p}$. To begin with, we deduce from the oddness of 
 			$y$ that $\displaystyle\int_0^Ly(x)dx=0$. Therefore, since $\varphi(L)=0$, we have $$ 	\mathtt{p}(L)=\frac{1}{\omega}\left(\int_0^Ly(x)dx\right)\varphi'(L)-\frac{1}{\omega}\varphi(L)y(L)=0=	\mathtt{p}(0).$$ At this point, given that 
 			$$\begin{array}{lllll}\displaystyle \omega 	\mathtt{p}'(x)
 				&=&\displaystyle\left(\int_0^xy(s)ds\right)\varphi''(x)-y'(x)\varphi(x),\end{array}$$
 			it follows that $	\mathtt{p}'(L)=	\mathtt{p}'(0)=0$ and $	\mathtt{p}$ is $L$-periodic as requested. So, considering $\tilde{f}=	\mathtt{p}$, we see that  $\tilde{f}$ satisfies the IVP
 			\begin{equation}\label{Cauchyproblem1} 
 				\begin{cases}
 					-\omega\tilde{f}''-\cosh(\varphi) \tilde{f}=1, \\
 					\tilde{f}(0)= 0, \\
 					\tilde{f}'(0)=0.
 				\end{cases}
 			\end{equation}
 			Problem \eqref{Cauchyproblem1} is appropriate for numerical calculations. Indeed, solving numerically  \eqref{Cauchyproblem1} for $L=\pi$, we find   $\displaystyle D_1=(\mathcal{L}_1^{-1}1,1)_{L^2_{per}}=(\tilde{f},1)_{L^2_{per}}=\int_{0}^{L}\tilde{f}(x)dx<0$ for all  $c\in (-1,1)$. A representation of the behavior of $D_1$ in this case, is presented in Figure \ref{fig12}.
 			
 			\begin{figure}[h] 
 				\centering
 				\includegraphics[scale=0.5,angle=0
 				]{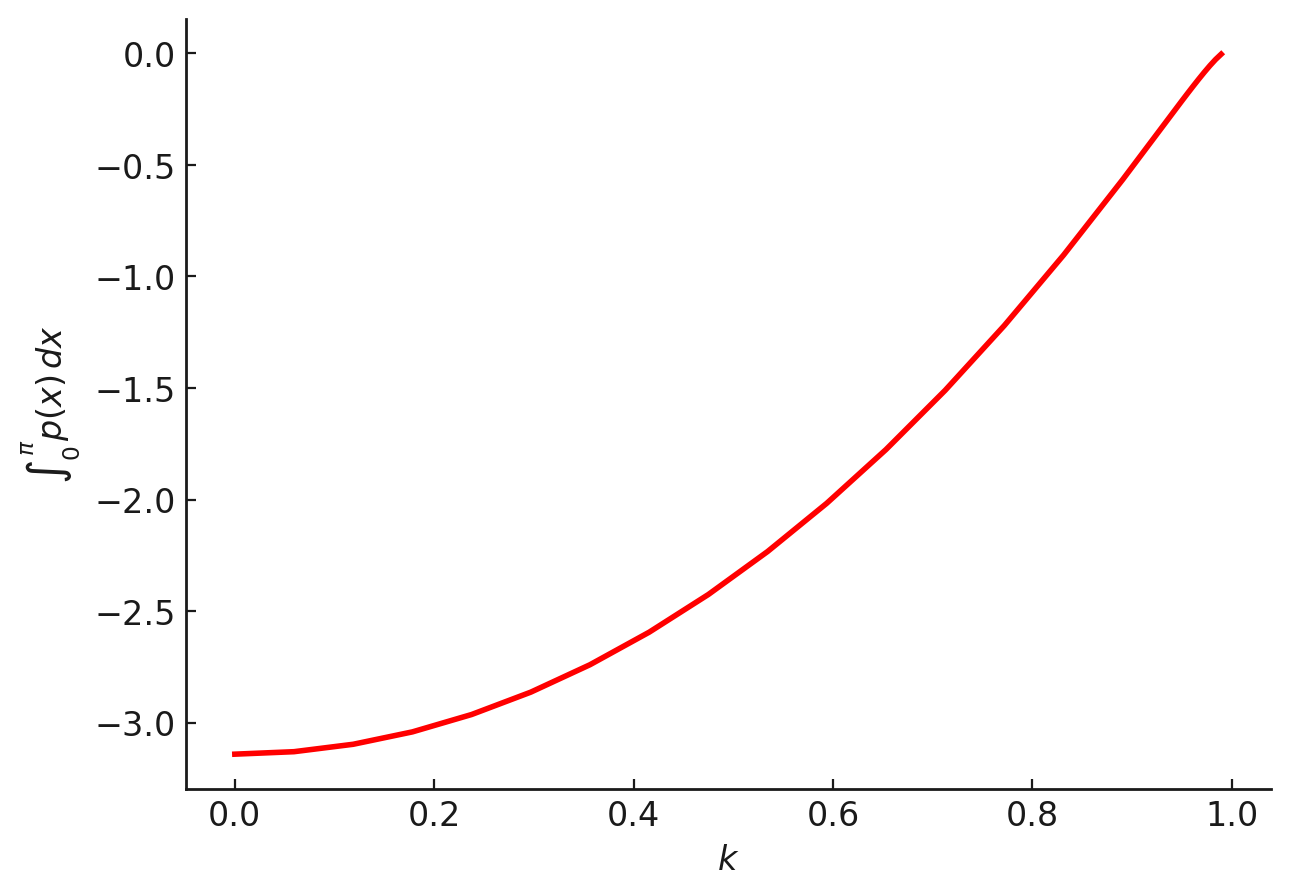} 
 				\caption{Graphic of $D_1$ for $L=\pi$.}\label{fig12}
 			\end{figure}
 		We can prove that $D_1 < 0$ for all $L \in (0,2\pi)$ by using a scaling argument.
 		 In fact, consider $\tilde{\varphi}(x)=\varphi\left(\frac{L}{\tilde{L}}x\right)$, where $\tilde{L}>0$ and $\varphi$ solves the equation \eqref{travSG1}. Since $\varphi$ is $L$-periodic, we obtain that $\tilde{\varphi}$ is $\tilde{L}$-periodic and solves the equation  
 			\begin{equation}\label{travSG1'}
 				-\tilde{\omega} \tilde{\varphi} ''-\sinh(\tilde{\varphi})=0, 
 			\end{equation}where $\tilde{\omega} =\left(\frac{\tilde{L}}{L}\right)^2\omega$. Next, by considering the new solution $\tilde{\varphi}$ instead of $\varphi$ in the initial value problem \eqref{Cauchyproblem0}, we obtain the non-periodic solution $\tilde{y}(x)=y\left(\frac{L}{\tilde{L}}x\right)$ that solves the problem  
 			\begin{equation}\label{Cauchyproblem'} 
 				\begin{cases}
 					- \tilde{y}''-\frac{1}{\tilde{\omega}}\cosh(\tilde{\varphi}) \tilde{y}=0, \\
 					\tilde{y}(0)= 0, \\
 					\tilde{	y}'(0)=\frac{L}{\tilde{L}}\frac{1}{\varphi'(0)}.
 				\end{cases}
 			\end{equation}
 			
 			\indent We also obtain that $\tilde{	\mathtt{p}}=	\mathtt{p}\left(\frac{L}{\tilde{L}}x\right)$ satisfies the problem \eqref{Cauchyproblem1} with $\tilde{\omega}$ in place of $\omega$ and  $\tilde{\varphi}$ in place of $\varphi$. Then, doing the calculations as done before and considering $\tilde{L}=\pi$, we obtain 
 			\begin{equation*}
 				D_1=(\mathcal{L}_1^{-1}1,1)_{L^2_{per}}=(	\mathtt{p},1)_{L^2_{per}}=\int_{0}^{L} 	\mathtt{p}(s)ds=\frac{L}{\pi}\int_{0}^{\pi} \tilde{	\mathtt{p}}(s)ds <0,
 			\end{equation*} as desired.\\
 			\indent Thus, the Index Theorem stated in \cite[Theorem 5.3.2]{kapitula} and \cite[Theorem 4.1]{pel-book} yields
 			\begin{equation*}
 				\text{n}(\mathcal{L}_{1\Pi}) =2-1-0=1 \quad \mbox{and}\quad  	\text{z}(\mathcal{L}_{1\Pi})=1+0=1.
 			\end{equation*}
 			
 			Similarly, the Index Theorem applied to the operators $\mathcal{L}_{\Pi}$ and $\mathcal{L}$ yields the relations
 			\begin{equation}\label{indexformula120}
 				\operatorname{n}(\mathcal{L}_{{\Pi}})=\text{n}(\mathcal{L})-{\rm N}_0-{\rm Z}_0 \quad \text{and} \quad \text{z}(\mathcal{L}_{{\Pi}})=\text{z}(\mathcal{L})+{\rm Z}_0,
 			\end{equation}
 			where the numbers ${\rm N}_0$ and ${\rm Z}_0$ are defined respectively as 
 			\begin{equation}\label{N0Z0}
 				{\rm N}_0=
 				\begin{cases}
 					1, \: \text{if} \: D_1 L<0, \\
 					0, \: \text{if} \: D_1 L\geq 0,\ \
 				\end{cases}
 				\quad \text{and} \qquad\ 
 				{\rm Z}_0=
 				\begin{cases}
 					1, \: \text{if} \: D_1L=0, \\
 					0, \: \text{if} \: D_1 L \neq 0.
 				\end{cases}
 			\end{equation}
 			Hence, ${\rm N}_0=1$ and ${\rm Z}_0=0$, proving that $\operatorname{n}(\mathcal{L}_{\Pi})=1$ and $\operatorname{z}(\mathcal{L}_{\Pi})=1$, as desired. 
 		\end{proof}

 \section{Spectral stability and instability of periodic waves for the sinh-Gordon equation}\label{section5}
 The purpose of this section is to investigate the spectral stability and instability of the periodic traveling wave solutions $\varphi$ given in \eqref{existSG1} for the sinh–Gordon equation \eqref{KF2}, based on the theory developed in \cite{kapitula,natali-thiago}. We have seen that equation \eqref{KF2} can be written, for $\alpha=-1$, in the form of a nonlinear Hamiltonian system as
 \begin{align}\label{hamilt310}
 	\begin{cases}
 		&	U_t=J \mathcal{E}'(U),\quad t>0,\\ 
 		&U|_{t=0}=U_0,
 	\end{cases}	
 \end{align}   where $U=(u,u_t)=(u,v)$, $J$ is the skew-symmetric matrix
 \begin{equation}\label{matrixJ}
 	J=\left(
 	\begin{array}{cc}
 		0 & 1\\
 		-1 & 0\end{array}
 	\right),
 \end{equation} \noindent and $\mathcal{E}'$ denotes the Fr\'echet derivative of the energy $\mathcal{E}$ defined in \eqref{E}.
 
 In order to formulate our problem, we apply the result from Proposition $\ref{lwpthm}$ by working with the strong solution $u$ of the problem $(\ref{CPKG3})$. In particular, this strong solution $u$ is, in particular, a \textit{weak solution} of the problem $(\ref{CPKG})$ as we have mentioned in Remark $\ref{remweakstrong}$. In fact,  let us consider $\mathtt{v}=\partial_x u_t$, that is, $u_t =\partial_x^{-1 }\mathtt{v}$. Then,
 \begin{equation}\label{hamilt310'}
 	\mathtt{v}_t=\partial_x u_{tt}=\partial_x( \partial_x^2u + \sinh(u)). 
 \end{equation} We define the perturbations of the traveling waves as
  \begin{align}\label{hamilt311}
 	\begin{cases}
 		&p(x+ct,t)=u(x,t)-\varphi(x+ct), \\ 
 		&q(x+ct,t)=\mathtt{v}(x,t)-c\varphi''(x+ct).
 	\end{cases}	
 \end{align}  
 By differentiating the first equation in \eqref{hamilt311} with respect to $t$ and equating it to the result obtained by applying the operator $\partial_x^{-1}$ to the second one, we obtain 
 \begin{equation}\label{hamilt312}
 	p_t=-c\partial_x p+\partial_x^{-1 }q.
 \end{equation}
Proceeding with the differentiation of the second equation in \eqref{hamilt311} with respect to $t$, we have that 
 \begin{align}\label{hamilt313}
 	\mathtt{v}_t=c\partial_x q+q_t+c^2 \varphi'''.
 \end{align}
 From \eqref{hamilt310'} and \eqref{hamilt313} it follows that 
 \begin{equation}\label{hamilt314}\begin{array}{llllll}
 	q_t&=&\displaystyle-c\partial_x q+\partial_x^3(p+\varphi) +  \partial_x(\sinh(p+\varphi))-c^2 \varphi'''\\\\
 	&=&\displaystyle-c\partial_x q+\partial_x^3 p+\partial_x(\sinh(p+\varphi))+\omega\varphi '''. 
 \end{array}\end{equation}
Upon differentiating \eqref{KF3} with respect to $x$ and substituting the result into \eqref{hamilt314}, we obtain that
 \begin{align}\label{hamilt315}
 	q_t=&-c\partial_x q+\partial_x^3 p+\partial_x(\sinh(p+\varphi)-\sinh(\varphi)). 
 \end{align}
 Using the Taylor expansion for the function $\sinh(p)$ around the point $\varphi$, one finds that 
 \begin{align*}
 	\sinh(p+\varphi)=\sinh(\varphi)+\cosh(\varphi)p +\mathcal{O}(p^2), 
 \end{align*}that is, 
 \begin{align*}
 	\sinh(p+\varphi)-\sinh(\varphi)\approx\cosh(\varphi)p .
 \end{align*}
 Substituting this linearization into \eqref{hamilt315}, we obtain the linear system in the form 
 \begin{equation}
 	\left\{\begin{array}{lll}
 		 p_t =-c\partial_x p+\partial_x^{-1 }q, \\ \\
 	     q_t =-c\partial_x q+\partial_x^3 p+\partial_x(\cosh(\varphi)p),
 	\end{array}\right.\label{syspert}	
 \end{equation} or equivalently in matrix form
 \begin{align}\label{hamilt316}
 	\begin{pmatrix}
 		p \\ q
 	\end{pmatrix}_t
 	=
 	\begin{pmatrix}
 		\begin{array}{ccc}
 			-c \, \partial_x && \partial_x^{-1} \\
 			\partial_x^3 + \partial_x(\cosh(\varphi)\cdot)  && -c \, \partial_x
 		\end{array}
 	\end{pmatrix}
 	\begin{pmatrix}
 		p \\ q
 	\end{pmatrix}.
 \end{align} 
 \indent Next, let us consider the growing mode solutions given by $p(x,t)=e^{\lambda t}g(x)$ and $q(x,t)=e^{\lambda t}h(x)$. After substituting into the system $(\ref{hamilt316})$, we obtain the spectral problem
 \begin{align}\label{hamilt3165}
 	\begin{pmatrix}
 		\begin{array}{ccc}
 			-c \, \partial_x && \partial_x^{-1} \\
 			\partial_x^3 + \partial_x(\cosh(\varphi)\cdot)  && -c \, \partial_x
 		\end{array}
 	\end{pmatrix}
 	\begin{pmatrix}
 		g \\ h
 	\end{pmatrix}=\lambda\begin{pmatrix}
 		g \\ h
 	\end{pmatrix}.
 	 \end{align} 
 Because of the presence of the derivative in the hyperbolic function, we see that $(\ref{hamilt3165})$ is then equivalent to the following projected problem
 \begin{align}\label{hamilt31651}
 	\begin{pmatrix}
 		\begin{array}{ccc}
 			-c \, \partial_x && \partial_x^{-1} \\
 			\partial_x^3 + \partial_x\left(\cosh(\varphi)\cdot-\frac{1}{L}\int_0^L\cosh(\varphi)\cdot dx\right)  && -c \, \partial_x
 		\end{array}
 	\end{pmatrix}
 	\begin{pmatrix}
 		g \\ h
 	\end{pmatrix}=\lambda\begin{pmatrix}
 		g \\ h
 	\end{pmatrix}.
 	 \end{align} 
 System $(\ref{hamilt31651})$ can be reduced in the projected \textit{quadratic pencil} given by
 \begin{equation}\label{quadrapencil}
 \lambda^2g+2c\lambda \partial_xg-(1-c^2)\partial_x^2g-\cosh(\varphi)g+\frac{1}{L}\int_0^L\cosh(\varphi)gdx=0,
 \end{equation}
 which is related with the projected linearization of the Hamiltonian structure $(\ref{hamilt310})$ given by
 \begin{equation}\label{hamilt2.0}
 		 J\mathcal{L}_{\Pi}V=\lambda V,
 	 \end{equation}  where $V=\begin{pmatrix}
 g\\
\tilde{h}
 \end{pmatrix}.$ It is important to note that the associated quadratic pencil for the spectral problem \eqref{hamilt2.0} is also given by \eqref{quadrapencil}. We now introduce the following definition of spectral stability:
 
 \begin{defi}\label{def.spc.ins}
 	The periodic wave $(\varphi, c\varphi')$ is said to be spectrally stable if the quadratic pencil in \eqref{quadrapencil} is solvable and all corresponding values of $\lambda$ are purely imaginary. Otherwise, that is, if there exists some $\lambda$ with $\operatorname{Re}(\lambda)>0$ and $u\not\equiv 0$ such that the pair $(\lambda,u)$ satisfies \eqref{quadrapencil}, we say that the periodic wave $(\varphi, c\varphi')$ is spectrally unstable.
 \end{defi}
 
 \begin{remark}\label{remspecstab}
 	In the context of classical spectral stability definitions, we say that $(\varphi, c\varphi')$ is spectrally stable if the spectrum of $J\mathcal{L}_{\Pi}$ is contained in the imaginary axis of the complex plane. Alternatively, $(\varphi,c\varphi')$ is said to be spectrally unstable if the spectrum of $J\mathcal{L}_{\Pi}$ contains an eigenvalue with positive real part. Both definitions are equivalent in the space $\mathbb{L}_{per,m}^2$.
 \end{remark}

According to Remark $\ref{remspecstab}$, we need to study the spectrum of $J\mathcal{L}_{\Pi}$. In fact, by Remark $\ref{smoothcomega}$, we can differentiate equation $(\ref{travSG1})$ with respect to $c\in(-1,1)$ to deduce
 \begin{equation}\label{L1c}
 	\mathcal{L}_1\left(\frac{\partial\varphi}{\partial c}\right)=-2c\varphi''.
 \end{equation}   Furthermore, a simple calculation also gives
 \begin{equation}\label{calcL1}
 	\mathcal{L}\left(\frac{\partial\varphi}{\partial c},\frac{\partial}{\partial c}(c\varphi')\right)=\left(\mathcal{L}_1 \left(\frac{\partial\varphi}{\partial c}\right)+c\varphi'', \varphi '\right)
 	=(-c\varphi'',\varphi'). 
 \end{equation} Since both  $\frac{\partial\varphi}{\partial c}$ and $\frac{\partial}{\partial c}(c\varphi')$ have the zero mean property, we have, in particular, that
 \begin{equation}\label{calcL2}
 	\mathcal{L}_{\Pi}\left(\frac{\partial\varphi}{\partial c},\frac{\partial}{\partial c}(c\varphi ')\right)
 	=(-c\varphi'',\varphi'). 
 \end{equation}
 
 Next, we consider the $3\times 3$ matrix $\mathcal{D}$ defined as 
 \begin{equation}\label{matrixD'} 
 	\mathcal{D}=\left[\begin{array}{llllll}(\mathcal{L}^{-1}e_0,e_0)_{\mathbb{L}^2_{per}}& & (\mathcal{L}^{-1}e_1,e_0)_{\mathbb{L}^2_{per}}& & (\mathcal{L}^{-1}e_2,e_0)_{\mathbb{L}^2_{per}}\\\\
 		(\mathcal{L}^{-1}e_0,e_1)_{\mathbb{L}^2_{per}}& & (\mathcal{L}^{-1}e_1,e_1)_{\mathbb{L}^2_{per}}& & (\mathcal{L}^{-1}e_2,e_1)_{\mathbb{L}^2_{per}}\\\\
 		(\mathcal{L}^{-1}e_0,e_2)_{\mathbb{L}^2_{per}}& & (\mathcal{L}^{-1}e_1,e_2)_{\mathbb{L}^2_{per}}& & (\mathcal{L}^{-1}e_2,e_2)_{\mathbb{L}^2_{per}}\end{array}\right],
 \end{equation}where $e_0=(-c\varphi '', \varphi')$, $e_1=(1,0)$ and $e_2=(0,1)$.
 
We have that $\mathcal{L}$ is self-adjoint and $\Ker(\mathcal{L}) = [(\varphi',c\varphi'')]$. Furthermore, with respect to the inner product in $\mathbb{L}^2_{per}$, the set $\{e_0,e_1,e_2\}$ is contained in $\Ker(\mathcal{L})^\perp=\operatorname{Range}(\mathcal{L})$. This ensures that the inverse operator $\mathcal{L}^{-1}$, which appears in the entries of the matrix $\mathcal{D}$, is well-defined, since it acts on the restricted domain and codomain $\mathcal{L}^{-1} :\Ker(\mathcal{L})^\perp\rightarrow \Ker(\mathcal{L})^\perp$. To proceed with the spectral analysis and relate it to the matrix $\mathcal{D}$, it is necessary to introduce some important concepts.
 
 \begin{defi}\label{de1.krein}
 We call the number of positive real eigenvalues of $J \mathcal{L}_{\Pi}$ (up to multiplicity) the real Krein index, denoted by $k_r$, and the number of eigenvalues of $J \mathcal{L}_{\Pi}$ with a nonzero imaginary part and a positive real part (again up to multiplicity), the complex Krein index, denoted by $k_c$.
 \end{defi}
 
 \begin{defi}\label{de2.krein}
 	Let $\lambda$ be a purely imaginary nonzero eigenvalue of the operator $J \mathcal{L}_{\Pi}$
 	with associated generalized eigenspace $E_\lambda= [v_1^\lambda, \dots, v_m^\lambda]$. We define the Krein signature of $\lambda$ by
 	\begin{equation*}
 		k_i^{-}(\lambda)=\operatorname{n} \left((w,(\mathcal{L} |_{E_\lambda})w)_{L^2_{per}}  \right), 
 	\end{equation*}where $w \in E_\lambda$ and $\operatorname{n} \left((w,(\mathcal{L}|_{E_\lambda})w)_{L^2_{per}}  \right)$ denotes the dimension of the maximal subspace for which  $(w,(\mathcal{L}|_{E_\lambda})w)_{L^2_{per}} <0$. If $k_i^-(\lambda) \ge 1$, then the eigenvalue $\lambda$ is said to have a negative Krein signature; otherwise, it is said to have a positive Krein signature. Furthermore, we define the total Krein signature by the quantity
 	\begin{equation*}
 		k_i^- = \sum_{ \lambda \neq 0, \ \lambda\in (i\mathbb{R})} k_i^-(\lambda).
 	\end{equation*}
 	  \end{defi}
 
 \begin{defi}\label{Hamk}
 	The Hamiltonian–Krein index is given by the sum
 	$$ 	\mathcal{K}_{Ham} =k_r+k_c+k_i^-$$
 	of the real, complex, and negative Krein indices, defined in Definitions \eqref{de1.krein} and \eqref{de2.krein}.
 \end{defi}

With the previously established results, we are able to determine the intervals of spectral stability and instability for the projected Cauchy problem, according to Definition \ref{def.spc.ins}. \\

\noindent\textit{Proof of the Theorem \ref{stabthm}.} Since $\Ker(\mathcal{L})=\Ker(\mathcal{L}_{\Pi})=[(\varphi',c\varphi'')]$, it follows from Definition $\ref{Hamk}$ and \cite[Theorem 7.1.5]{kapitula} that the periodic wave $(\varphi, c\varphi')$ is spectrally stable if the quantity
\begin{equation}\label{stab.condition}
	\mathcal{K}_{Ham}=\operatorname{n}(\mathcal{L}) - \operatorname{n}(\mathcal{D})
\end{equation}
is zero. Now, if the difference in \eqref{stab.condition} is an odd number, the periodic wave is spectrally unstable.
 
According to the proof of Proposition $\ref{leman1}$ and the identity $(\mathcal{L}^{-1}e_1,e_1)_{\mathbb{L}^2{per}}=D_1$, determining the remaining entries of the matrix $\mathcal{D}$ in \eqref{matrixD'} requires calculating the following inner products:
 $(\mathcal{L}^{-1}e_0,e_0)_{\mathbb{L}^2_{per}}$, $(\mathcal{L}^{-1}e_1,e_0)_{\mathbb{L}^2_{per}}=	(\mathcal{L}^{-1}e_0,e_1)_{\mathbb{L}^2_{per}}$, and $(\mathcal{L}^{-1}e_2,e_0)_{\mathbb{L}^2_{per}}=		(\mathcal{L}^{-1}e_0,e_2)_{\mathbb{L}^2_{per}}$. In fact, it follows from \eqref{calcL1} that \begin{equation}\nonumber
 	\mathcal{L}^{-1}e_0=	 \left(\frac{\partial\varphi}{\partial c},\frac{\partial}{\partial c}(c\varphi')\right) .
 \end{equation}
 Then, after some computations and using the fact that the function $\varphi'$ has zero mean, the function $\frac{\partial\varphi}{\partial c}$ is odd, and $\frac{\partial\varphi'}{\partial c}$ is $L$-periodic, we obtain that 
 \begin{align}\label{eqD1}
 	(\mathcal{L}^{-1}e_0,e_0)_{\mathbb{L}^2_{per}} =\int_{0}^{L}\left[(\varphi')^2- 2c\varphi''\frac{\partial\varphi}{\partial c}\right]dx=\int_{0}^{L}\left[(\varphi')^2+ c\frac{\partial(\varphi')^2}{\partial c}\right]dx=\frac{\partial}{\partial c}\int_0^L[c(\varphi')^2]dx,
 \end{align}
 
 \begin{align}\label{eqD2}
 	(\mathcal{L}^{-1}e_0,e_1)_{\mathbb{L}^2_{per}}=
 	\int_{0}^{L}\frac{\partial\varphi}{\partial c} dx=0,
 \end{align}
 and 
 \begin{align}\label{eqD3}
 	(\mathcal{L}^{-1}e_0,e_2)_{\mathbb{L}^2_{per}}=
 	\int_{0}^{L}\left[\varphi'+c\left(\frac{\partial\varphi'}{\partial c}\right)\right] dx=0.
 \end{align}
 \indent Using equations \eqref{eqD1}, \eqref{eqD2}, and \eqref{eqD3}, the matrix $\mathcal{D}$ in \eqref{matrixD'} can be simplified to
 \begin{equation}\label{matrixD''}
 	\mathcal{D}=\left[\begin{array}{llllll} \displaystyle \frac{\partial}{\partial c}\int_0^L[c(\varphi')^2]dx
 		& &  0& & 	 	0\\\\0& &D_1& & 0\\\\
 		0& & 0& & L\end{array}\right],
 \end{equation}
 where $D_1=(\mathcal{L}_{1}^{-1}1,1)_{L_{per}^2}$. Our aim is to determine, for fixed values of $L \in (0, 2\pi)$, the sign of $\det(\mathcal{D})$ for all values of $c\in (-1,1)$. Indeed, we obtain by applying the chain rule to the integral term in \eqref{matrixD''} that
 \begin{align}\label{detD}
 	\det(\mathcal{D})=&D_1 L \frac{\partial}{\partial c}\int_{0}^{L}[c(\varphi')^2]dx=D_1L\left[\int_{0}^{L}(\varphi')^2dx-2c^2\left(\frac{d\omega}{dk}\right)^{-1}\frac{\partial}{\partial k}\int_{0}^{L}(\varphi')^2dx\right].
 \end{align}
 From $(\ref{varphilinha})$, using properties of elliptic integrals and the relation $b=\frac{4K(k)}{L}$, we have that
 \begin{equation}\label{intvarphilinha}
 \int_{0}^{L}(\varphi')^2dx=64k^2\frac{K(k)}{L}\int_0^{K(k)}\frac{{\rm cn}(x;k)^2}{{\rm dn}(x;k)^2}dx=\frac{64K(k)(K(k)-E(k))}{L}.
 \end{equation}
 Since $\omega$ is given by \eqref{speed2}, from \eqref{detD}, \eqref{intvarphilinha}, and Formulas (710.00) and (710.05) in \cite{byrd} we obtain 
 
 \begin{equation}\label{detD1}\begin{array}{llll}
 \det(\mathcal{D}) 
 &=&D_1L\displaystyle \frac{64K(k)^3}{L^3(K(k)-E(k))}\left\{k^2L^{2} -16(1-k^{2})  \left[k^{2}K(k)^{2} - (K(k)-E(k))^{2}\right] \right\} \\\\ 
&=&\displaystyle D_1L\frac{64K(k)^3}{L^3(K(k)-E(k))}\mathsf{r}(k,L).
 \end{array}
 \end{equation}
 
As $L>0$, $D_1<0$, and $\displaystyle\frac{64K(k)^3}{L^3(K(k)-E(k))}>0$ for all $k\in (0,1)$, we need to analyze the sign of the function $\mathsf{r}(k,L) $ in terms of $k\in (0,1)$ and for a fixed $L\in (0,2\pi)$.  In fact, it is well known that the complete elliptic integrals satisfy 
$\displaystyle K(k),\, E(k) \to \frac{\pi}{2}$ as $k \to 0^{+}$, $K(k) \to +\infty$ as $k \to 1^{-}$, 
and $E(k) \to 1 $ as $k \to 1^{-}$. With these convergences and by applying L’Hôpital’s rule, we have 
 $$\displaystyle\lim_{k\rightarrow1^-}\left[(1-k^2)K (k)  \right]=\lim_{k\rightarrow1^-}\left[(1-k^2)K (k)^2 \right]=0.$$
Thus, for each fixed $L\in(0,2\pi)$, the function $\mathsf{r}$ tends to $0$ as $k$ approaches $0^+$, and to $L^{2}$ as $k$ approaches $1^{-}$. In addition, this function is smooth and changes its sign in a unique point $k_0=k_0(L)\in(0,1)$. In fact, differentiating the expression of $\mathsf{r}$ with respect to $k$, we obtain
\begin{align}\label{detD1.eq2}\nonumber
	\frac{\partial\mathsf{r}}{\partial k}= 2k\left\{   L^2-16 \left[ (1-k^2)K(k)^2-2E(k)K(k)+2E(k)^2\right]\right\} 
	=2k\left\{   L^2-16\mathsf{f}(k)\right\}.
\end{align}
Since $\displaystyle\lim_{k\rightarrow 0^{+}}\frac{\partial\mathsf{r}}{\partial k}=L^2-4\pi^2<0$ and $\displaystyle\lim_{k\rightarrow 1^{-}}\frac{\partial\mathsf{r}}{\partial k}=+\infty$, we obtain, by the intermediate value theorem, the existence of $k_1=k_1(L)\in(0,1)$ such that  $\displaystyle \frac{\partial\mathsf{r}}{\partial k}\Big|_{k=k_1}=0$. On the other hand, function $\displaystyle\frac{\partial\mathsf{r}}{\partial k}$ is strictly increasing in terms of elliptic modulus $k$. This implies that $k_1$ is the only critical point of $\mathsf{r}$ over the interval $(0,1)$. This means, for each fixed $L\in (0,2\pi)$, that there 
exists a unique $k_{0}$, depending on $L$, such that $\mathsf{r}(k_{0},L)=0$, 
$\mathsf{r}(k,L)<0$ for $k\in(0,k_{0})$, and $\mathsf{r}(k,L)>0$ for 
$k\in(k_{0},1)$. Moreover, it is important to emphasize that the critical point $k_1$ of $\mathsf{r}$ moves closer to $1$ as $L$ approaches $2\pi$, and consequently, the same holds for the root $k_0$. Below, we illustrate the behavior of 
$\mathsf{r}$ for certain values of $L\in(0,2\pi)$.

\begin{figure}[h] 
	\centering
	\includegraphics[scale=0.5,angle=0
	]{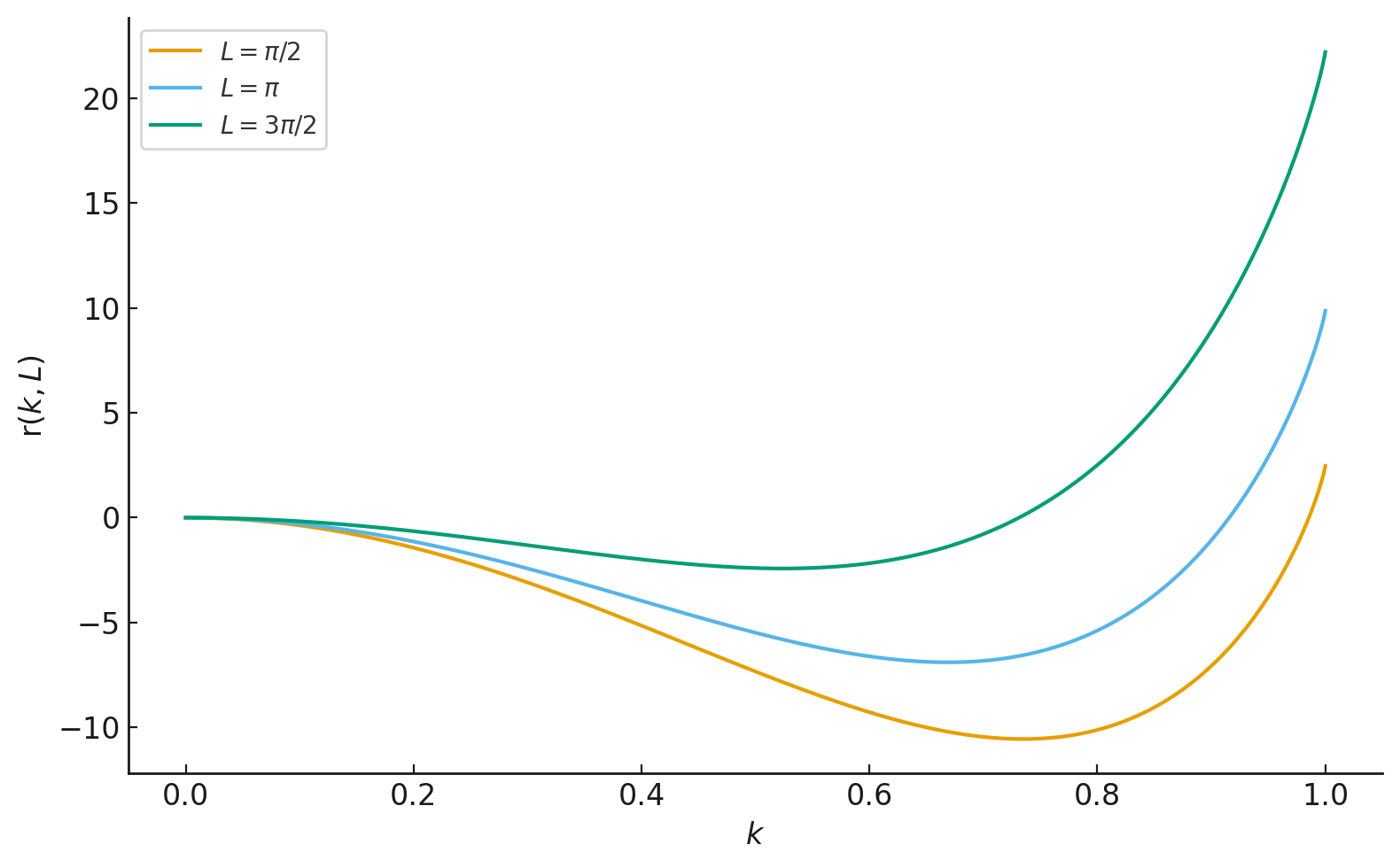} 	\caption{Graph of $\mathsf{r}(k,L)$ for $L=\frac{\pi}{2},\pi, \frac{3\pi}{2}$.}\label{rkgraphic}
\end{figure}

 Now, since the determinant of $\mathcal{D}$ is the product of its three eigenvalues, $L>0$, and $D_1<0$, it follows that
 \begin{align*}
 	\operatorname{n}(\mathcal{D})= \begin{cases}
 	 	&  2, \quad \text{if} \  k\in (0,k_0), \\ 
 	 &  1, \quad \text{if} \ k\in [k_0,1). 
 	 \end{cases}
 \end{align*}
  Therefore, the Hamiltonian-Krein index becomes
 \begin{align*}
 	\mathcal{K}_{Ham}=	\operatorname{n}(\mathcal{L}) - \operatorname{n}(\mathcal{D})=\begin{cases}
 		&  0, \quad \text{if} \  k\in (0,k_0), \\ 
 		&  1, \quad \text{if} \ k\in [k_0,1). 
 	\end{cases}
 \end{align*} 
This completes the proof of the theorem.
 \hfill $\blacksquare$

\end{document}